\documentclass[11pt]{amsart}
\usepackage{amsmath,amssymb,amsthm} 
\usepackage[alphabetic,abbrev,nobysame]{amsrefs} 
\usepackage{tikz,tikz-cd}
\usepackage{xcolor}
\usepackage{mathtools}
\usepackage{stmaryrd} 
\DeclarePairedDelimiter\Brackets{\llbracket}{\rrbracket} 

\newcommand{\Z}{\mathbb{Z}} 
\newcommand{\Ptn}{\mathsf{P}} 
\newcommand{\TL}{\mathsf{TL}} 
\newcommand{\B}{\mathsf{B}} 
\newcommand{\M}{\mathsf{M}} 
\newcommand{\PTL}{\mathsf{PTL}} 
\newcommand{\PB}{\mathsf{PB}} 
\newcommand{\RP}{\mathsf{RP}} 
\newcommand{\LP}{\mathsf{LP}} 
\newcommand{\RR}{\mathsf{R}} 
\newcommand{\LL}{\mathsf{L}} 
\newcommand{\RL}{\mathsf{RL}} 

\newcommand{\End}{\operatorname{End}} 
\newcommand{\Hom}{\operatorname{Hom}} 
\newcommand{\gl}{\mathfrak{gl}} 
\newcommand{\fraksl}{\mathfrak{sl}} 
\newcommand{\ov}{\overline} 
\newcommand{\UU}{\mathbf{U}} 
\newcommand{\D}{\mathcal{D}} 

\newcommand{\e}{\varepsilon}

\newcommand{\topp}[1]{{\lceil #1 \rceil}}
\newcommand{\bott}[1]{{\lfloor #1 \rfloor}}

\newcommand{\Mat}{\operatorname{Mat}}
\newcommand{\rank}{\operatorname{rank}}
\newcommand{\shade}[1]{\fboxsep=0pt\colorbox{black!6}{$#1$}}
\newcommand{\bil}[2]{\langle #1, #2 \rangle} 

\newcommand{\dom}{\unrhd}  
\newcommand{\sdom}{\rhd}  
\newcommand{\lessdom}{\unlhd}  
\newcommand{\slessdom}{\lhd}  

\newtheorem{thm}{Theorem}[section]
\newtheorem*{thm*}{Theorem}
\newtheorem{lem}[thm]{Lemma}
\newtheorem*{lem*}{Lemma}
\newtheorem{prop}[thm]{Proposition}
\newtheorem*{prop*}{Proposition}
\newtheorem{cor}[thm]{Corollary}
\newtheorem*{cor*}{Corollary}

\newtheorem*{conj*}{Conjecture}

\theoremstyle{definition}

\newtheorem*{defn*}{Definition}

\newtheorem*{example*}{Example}
\newtheorem{rmk}[thm]{Remark}
\newtheorem*{rmk*}{Remark}

\newtheorem*{que*}{Question}

\newcommand{\onepic}{
\begin{tikzpicture}[scale = 0.35,thick, baseline={(0,-1ex/2)}] 
\tikzstyle{vertex} = [shape = circle, minimum size = 4pt, inner sep = 1pt] 
\node[vertex] (G--1) at (0.0, -1) [shape = circle, draw,fill=black] {}; 
\node[vertex] (G-1) at (0.0, 1) [shape = circle, draw,fill=black] {}; 
\draw (G-1) .. controls +(0, -1) and +(0, 1) .. (G--1); 
\end{tikzpicture}
}

\newcommand{\ppic}{
\begin{tikzpicture}[scale = 0.35,thick, baseline={(0,-1ex/2)}] 
\tikzstyle{vertex} = [shape = circle, minimum size = 4pt, inner sep = 1pt] 
\node[vertex] (G--1) at (0.0, -1) [shape = circle, draw,fill=black] {}; 
\node[vertex] (G-1) at (0.0, 1) [shape = circle, draw,fill=black] {}; 
\end{tikzpicture}
}

\newcommand{\spic}{
\begin{tikzpicture}[scale = 0.35,thick, baseline={(0,-1ex/2)}] 
\tikzstyle{vertex} = [shape = circle, minimum size = 4pt, inner sep = 1pt] 
\node[vertex] (G--2) at (1.5, -1) [shape = circle, draw,fill=black] {}; 
\node[vertex] (G-1) at (0.0, 1) [shape = circle, draw,fill=black] {}; 
\node[vertex] (G--1) at (0.0, -1) [shape = circle, draw,fill=black] {}; 
\node[vertex] (G-2) at (1.5, 1) [shape = circle, draw,fill=black] {}; 
\draw (G-1) .. controls +(0.75, -1) and +(-0.75, 1) .. (G--2); 
\draw (G-2) .. controls +(-0.75, -1) and +(0.75, 1) .. (G--1); 
\end{tikzpicture}
}

\newcommand{\bpic}{
\begin{tikzpicture}[scale = 0.35,thick, baseline={(0,-1ex/2)}]
\tikzstyle{vertex} = [shape = circle, minimum size = 4pt, inner sep = 1pt] 
\node[vertex] (G--2) at (1.5, -1) [shape = circle, draw,fill=black] {}; 
\node[vertex] (G--1) at (0.0, -1) [shape = circle, draw,fill=black] {}; 
\node[vertex] (G-1) at (0.0, 1) [shape = circle, draw,fill=black] {}; 
\node[vertex] (G-2) at (1.5, 1) [shape = circle, draw,fill=black] {}; 
\draw (G-1) .. controls +(0.5, -0.5) and +(-0.5, -0.5) .. (G-2); 
\draw (G-2) .. controls +(0, -1) and +(0, 1) .. (G--2); 
\draw (G--2) .. controls +(-0.5, 0.5) and +(0.5, 0.5) .. (G--1); 
\draw (G--1) .. controls +(0, 1) and +(0, -1) .. (G-1); 
\end{tikzpicture} 
}

\newcommand{\epic}{
\begin{tikzpicture}[scale = 0.35,thick, baseline={(0,-1ex/2)}]
\tikzstyle{vertex} = [shape = circle, minimum size = 4pt, inner sep = 1pt] 
\node[vertex] (G--2) at (1.5, -1) [shape = circle, draw,fill=black] {}; 
\node[vertex] (G--1) at (0.0, -1) [shape = circle, draw,fill=black] {}; 
\node[vertex] (G-1) at (0.0, 1) [shape = circle, draw,fill=black] {}; 
\node[vertex] (G-2) at (1.5, 1) [shape = circle, draw,fill=black] {}; 
\draw (G--2) .. controls +(-0.5, 0.5) and +(0.5, 0.5) .. (G--1); 
\draw (G-1) .. controls +(0.5, -0.5) and +(-0.5, -0.5) .. (G-2); 
\end{tikzpicture} 
}

\newcommand{\lpic}{
\begin{tikzpicture}[scale = 0.35,thick, baseline={(0,-1ex/2)}] 
\tikzstyle{vertex} = [shape = circle, minimum size = 4pt, inner sep = 1pt] 
\node[vertex] (G--2) at (1.5, -1) [shape = circle, draw,fill=black] {}; 
\node[vertex] (G-1) at (0.0, 1) [shape = circle, draw,fill=black] {}; 
\node[vertex] (G--1) at (0.0, -1) [shape = circle, draw,fill=black] {}; 
\node[vertex] (G-2) at (1.5, 1) [shape = circle, draw,fill=black] {}; 
\draw (G-1) .. controls +(0.75, -1) and +(-0.75, 1) .. (G--2); 
\end{tikzpicture}
}

\newcommand{\rpic}{
\begin{tikzpicture}[scale = 0.35,thick, baseline={(0,-1ex/2)}] 
\tikzstyle{vertex} = [shape = circle, minimum size = 4pt, inner sep = 1pt] 
\node[vertex] (G--2) at (1.5, -1) [shape = circle, draw,fill=black] {}; 
\node[vertex] (G--1) at (0.0, -1) [shape = circle, draw,fill=black] {}; 
\node[vertex] (G-2) at (1.5, 1) [shape = circle, draw,fill=black] {}; 
\node[vertex] (G-1) at (0.0, 1) [shape = circle, draw,fill=black] {}; 
\draw (G-2) .. controls +(-0.75, -1) and +(0.75, 1) .. (G--1); 
\end{tikzpicture}
}

\newcommand{\oneonepic}{
\begin{tikzpicture}[scale = 0.35,thick, baseline={(0,-1ex/2)}] 
\tikzstyle{vertex} = [shape = circle, minimum size = 4pt, inner sep = 1pt] 
\node[vertex] (G--2) at (1.5, -1) [shape = circle, draw,fill=black] {}; 
\node[vertex] (G--1) at (0.0, -1) [shape = circle, draw,fill=black] {}; 
\node[vertex] (G-2) at (1.5, 1) [shape = circle, draw,fill=black] {}; 
\node[vertex] (G-1) at (0.0, 1) [shape = circle, draw,fill=black] {}; 
\draw (G-1) -- (G--1);
\draw (G-2) -- (G--2); 
\end{tikzpicture}
}

\newcommand{\oneppic}{
\begin{tikzpicture}[scale = 0.35,thick, baseline={(0,-1ex/2)}] 
\tikzstyle{vertex} = [shape = circle, minimum size = 4pt, inner sep = 1pt] 
\node[vertex] (G--2) at (1.5, -1) [shape = circle, draw,fill=black] {}; 
\node[vertex] (G--1) at (0.0, -1) [shape = circle, draw,fill=black] {}; 
\node[vertex] (G-2) at (1.5, 1) [shape = circle, draw,fill=black] {}; 
\node[vertex] (G-1) at (0.0, 1) [shape = circle, draw,fill=black] {}; 
\draw (G-1) -- (G--1);
\end{tikzpicture}
}

\newcommand{\ponepic}{
\begin{tikzpicture}[scale = 0.35,thick, baseline={(0,-1ex/2)}] 
\tikzstyle{vertex} = [shape = circle, minimum size = 4pt, inner sep = 1pt] 
\node[vertex] (G--2) at (1.5, -1) [shape = circle, draw,fill=black] {}; 
\node[vertex] (G--1) at (0.0, -1) [shape = circle, draw,fill=black] {}; 
\node[vertex] (G-2) at (1.5, 1) [shape = circle, draw,fill=black] {}; 
\node[vertex] (G-1) at (0.0, 1) [shape = circle, draw,fill=black] {}; 
\draw (G-2) -- (G--2); 
\end{tikzpicture}
}

\newcommand{\sink}{
\begin{tikzpicture}[scale = 0.35,thick, baseline={(0,-1ex/2)}] 
\tikzstyle{vertex} = [shape = circle, minimum size = 4pt, inner sep = 1pt] 
\node[vertex] (G--2) at (1.5, -1) [shape = circle, draw,fill=black] {}; 
\node[vertex] (G--1) at (0.0, -1) [shape = circle, draw,fill=black] {}; 
\node[vertex] (G-2) at (1.5, 1) [shape = circle, draw,fill=black] {}; 
\node[vertex] (G-1) at (0.0, 1) [shape = circle, draw,fill=black] {}; 
\end{tikzpicture}
}

\newcommand{\cuppic}{
\begin{tikzpicture}[scale = 0.35,thick, baseline={(0,-1ex/2)}]
\tikzstyle{vertex} = [shape = circle, minimum size = 4pt, inner sep = 1pt] 
\node[vertex] (G--2) at (1.5, -1) [shape = circle, draw,fill=black] {}; 
\node[vertex] (G--1) at (0.0, -1) [shape = circle, draw,fill=black] {}; 
\node[vertex] (G-1) at (0.0, 1) [shape = circle, draw,fill=black] {}; 
\node[vertex] (G-2) at (1.5, 1) [shape = circle, draw,fill=black] {}; 
\draw (G-1) .. controls +(0.5, -0.5) and +(-0.5, -0.5) .. (G-2); 
\end{tikzpicture} 
}

\newcommand{\cappic}{
\begin{tikzpicture}[scale = 0.35,thick, baseline={(0,-1ex/2)}]
\tikzstyle{vertex} = [shape = circle, minimum size = 4pt, inner sep = 1pt] 
\node[vertex] (G--2) at (1.5, -1) [shape = circle, draw,fill=black] {}; 
\node[vertex] (G--1) at (0.0, -1) [shape = circle, draw,fill=black] {}; 
\node[vertex] (G-1) at (0.0, 1) [shape = circle, draw,fill=black] {}; 
\node[vertex] (G-2) at (1.5, 1) [shape = circle, draw,fill=black] {}; 
\draw (G--2) .. controls +(-0.5, 0.5) and +(0.5, 0.5) .. (G--1); 
\end{tikzpicture} 
}

\newcommand{\Part}[1]{
 \foreach \x [count=\s from 1] in {#1}{
 	{\ifnum\s=1
		\draw (0,\s-1)--(\x,\s-1); 
		\fi}
   \draw (0,\s) to (\x,\s);
   \foreach \y in {0, ..., \x} {\draw (\y,\s)--(\y,\s-1);}
 }}

\def\UNIT{.18} \newcommand{\PART}[1]{
\begin{tikzpicture}[xscale=\UNIT, yscale=-\UNIT] 
	\Part{#1}
\end{tikzpicture}
}

\renewcommand{\labelenumi}{(\alph{enumi})}
\allowdisplaybreaks

\begin{document}
\title[The partial Temperley--Lieb algebra]%
      {The partial Temperley--Lieb algebra\\and its representations}
\author{Stephen Doty}
\email{doty@math.luc.edu, tonyg@math.luc.edu}
\author{Anthony Giaquinto}
\address{Department of Mathematics and Statistics,
  Loyola University Chicago, Chicago, IL 60660 USA}


\dedicatory{Dedicated to the memory of Georgia Benkart}

\begin{abstract}\noindent
We give a combinatorial description of a new diagram algebra, the
partial Temperley--Lieb algebra, arising as the generic centralizer
algebra $\End_{\UU_q(\gl_2)}(V^{\otimes k})$, where $V = V(0) \oplus
V(1)$ is the direct sum of the trivial and natural module for the
quantized enveloping algebra $\UU_q(\gl_2)$. It is a proper subalgebra
of the Motzkin algebra (the $\UU_q(\fraksl_2)$-centralizer) of Benkart
and Halverson.  We prove a version of Schur--Weyl duality for the new
algebras, and describe their generic representation theory.
\end{abstract}
\maketitle

\section*{Introduction}\noindent
The Temperley--Lieb algebra $\TL_k(\delta)$ arose in \cite{TL} in
connection with the Potts model in mathematical physics. It was
rediscovered by Vaughan Jones in his seminal work
\cites{J83,J85,J86,J87a,J87b} on subfactors, in the guise of a von
Neumann algebra, enabling spectacular applications to knot theory and
many subsequent developments (see e.g.~\cites{Kauff, Martin:book, GHJ,
  Westbury, Ridout-StAubin}). An important feature of these algebras
is that when the ground ring is a field and $\delta = \pm(q+q^{-1})$,
\begin{equation}\label{e:TL-iso}
  \TL_k(\delta) \cong \End_{\UU_q(\fraksl_2)}(V(1)^{\otimes k})
\end{equation}
for almost all values of the parameter $q$, where $V(1)$ is the
$2$-dimensional type-$\mathbf{1}$ simple $\UU_q(\fraksl_2)$-module
(the ``natural'' module). Kauffman \cite{K:87} (see also \cite{BW})
found the now standard realization of $\TL_k(\delta)$ in terms of
planar Brauer diagrams in the Brauer algebra.

The partial Brauer algebra $\PB_k(\delta,\delta')$, the span of all
partition $k$-diagrams with blocksize at most two, was studied in
\cites{HdM,MM}. It comes naturally equipped with two independent
parameters $\delta$, $\delta'$ of disparate topological
significance. In \cite{BH} (see also \cite{DEG}) Benkart and Halverson
introduced the Motzkin algebra $\M_k(\delta,\delta')$, the subalgebra
of $\PB_k(\delta,\delta')$ spanned by planar partial Brauer diagrams.
It is known \cite{DG:twin} that $ \PB_k(\delta,\delta') \cong
\PB_k(\delta,1) $ for any $\delta' \ne 0$, so it is natural to
restrict one's attention to $\PB_k(\delta,1)$ and $\M_k(\delta,1)$.
For simplicity, set $\M_k(\delta) := \M_k(\delta,1)$; this is the
context of \cite{BH}.  When the ground ring is a field and $\delta =
1\pm(q+q^{-1})$, they obtain an isomorphism
\begin{equation}\label{e:M-iso}
\M_k(\delta) \cong \End_{\UU_q(\fraksl_2)}(V^{\otimes k})
\end{equation}
for almost all values of $q$, where $V = V(1)\oplus V(0)$ is the
direct sum of the natural module $V(1)$ (as above) and the trivial
module $V(0)$. Actually, \eqref{e:M-iso} is proved in \cite{BH} only
for the case $\delta = 1 - (q+q^{-1})$, but it is easily extended to
the above statement. Notice that the right hand side of
\eqref{e:M-iso} is independent of the choice of sign, so
$\M_k(1+(q+q^{-1})) \cong \M_k(1-(q+q^{-1}))$.

The isomorphism \eqref{e:M-iso} was unfortunately misstated in
\cite{BH}, where $\UU_q(\gl_2)$ incorrectly appeared in place of
$\UU_q(\fraksl_2)$ on the right hand side. The two centralizers
differ; indeed, their dimensions don't agree (see \S\ref{s:VV}).

The purpose of this paper is to identify a subalgebra $\PTL_k(\delta)$
of $\M_k(\delta)$, the \emph{partial Temperley--Lieb algebra} of the
title, such that when $\delta = 1\pm(q+q^{-1})$,
\begin{equation}\label{e:PTL-iso}
\PTL_k(\delta) \cong \End_{\UU_q(\gl_2)}(V^{\otimes k})
\end{equation}
for almost all values of $q$. The algebra $\PTL_k(\delta)$ has a basis
indexed by the set of \emph{balanced} Motzkin $k$-diagrams (diagrams
with the same number of cups as caps). It is notable that the basis
elements are not diagrams; instead, each basis element is an
alternating sum of diagrams with a unique maximal balanced diagram as
leading term. As the right hand side of \eqref{e:PTL-iso} is again
independent of the choice of sign, we have $\PTL_k(1+(q+q^{-1})) \cong
\PTL_k(1-(q+q^{-1}))$.

The use of the adjective ``partial'' in describing the algebras
$\PTL_k(\delta)$ fits into a more general scheme of ``partialization''
that goes back at least to \cite{Mazor}.

When the ground ring is a field, all irreducible representations of a
semisimple cellular algebra are absolutely irreducible. In other
words, cellular algebras over a field are semisimple if and only if
they are split semisimple, so the adjective ``split'' is often omitted
in describing these algebras in the semisimple case. This applies to
the algebras $\TL_k(\delta)$, $\M_k(\delta)$, and $\PTL_k(\delta)$
appearing in this paper.

Our main results are as follows:
\begin{enumerate}\renewcommand{\labelenumi}{(\roman{enumi})}
\item In Theorem~\ref{t:bases}, we find two natural bases
  $\{\bar{d}\}$, $\{\tilde{d} \}$ of $\PTL_k(\delta)$, each indexed by
  the set of balanced Motzkin $k$-diagrams. Theorem~\ref{t:bar-mult}
  works out a multiplication rule for each basis.

\item We show that $\PTL_k(\delta)$ is an iterated inflation of
  Temperley--Lieb algebras, in the sense of
  \cites{KX:99,KX:01,Green-Paget}, hence is cellular in the sense of
  \cite{GL:96}; see Remark~\ref{r:Morita}. More precisely, in
  Theorem~\ref{t:Xn} we prove the stronger result that
  $\PTL_k(\delta)$ is Morita equivalent to a direct sum of
  Temperley--Lieb algebras with parameter $\delta-1$.

\item We construct the cell modules for $\PTL_k(\delta)$, see
  Theorem~\ref{t:PTL-cell}, and prove that when the ground ring is a
  field and $\TL_n(\delta-1)$ is semisimple for all $n\le k$
  then the same is true of $\PTL_k(\delta)$, see Theorem~\ref{t:ss}.

\item In Theorem~\ref{t:Motzkin-SWD}, we slightly extend the
  aforementioned Schur--Weyl duality result of \cite{BH}, showing that
  there are many choices in how to make $\M_k(\delta)$ act on
  $V^{\otimes k}$, all of which imply that $\M_k(\delta)$ is
  isomorphic to the $\UU_q(\fraksl_2)$-centralizer of $V^{\otimes k}$
  under suitable hypotheses.

\item For $\delta = 1 \pm (q+q^{-1})$, under suitable hypotheses, we
  define a faithful action of $\PTL_k(\delta)$ on $V^{\otimes k}$
  commuting with the action of $\UU_q(\gl_2)$, and prove the
  isomorphism \eqref{e:PTL-iso} in Theorem~\ref{t:SWD}. Thus we obtain
  a version of Schur--Weyl duality for $V^{\otimes k}$ regarded as a
  bimodule for $\PTL_k(\delta)$, $\UU_q(\gl_2)$.
\end{enumerate}
Finally, in Appendix A, we reinterpret results of \cite{GHJ} to obtain
a precise semisimplicity criterion for $\TL_k(\delta)$, when $\delta =
\pm(q+q^{-1})$ and the ground ring is a field.

\section{Some diagram algebras}\label{s:diagrams}\noindent
In this section, we define the standard diagram algebras needed for
this paper. Unless stated otherwise, we work over an arbitrary unital
commutative ring $\Bbbk$.

\subsection{Terminology}
Let $[k] := \{1, \dots, k\}$ and $[k]' := \{1', \dots, k'\}$.  The set
$\Ptn_k$ is the collection of all set partitions (equivalence
relations) on $[k] \cup [k]'$. If $d = \{B_1, \dots, B_l\}$ belongs to
$\Ptn_k$ where $B_1, \dots, B_l$ are pairwise disjoint, we call the
$B_i$ the \emph{blocks} of $d$. Typically, $d$ is depicted by a graph
on $2k$ vertices arranged in two parallel rows in a rectangle, with
vertices in the top (resp., bottom) row indexed by $[k]$ (resp.,
$[k]'$) in order from left to right. Edges are drawn in the interior
of the rectangle in any way such that the resulting connected
components coincide with the blocks. Although this graphical depiction
of elements of $\Ptn_k$ is not in general unique, the lack of
uniqueness causes no difficulty.  To be precise, we define a
$k$-\emph{diagram} $d$ to be the equivalence class of all graphs
depicting its underlying set partition $d$, where two such depictions
are equivalent if and only if they have the same blocks.  Henceforth,
we identify elements of $\Ptn_k$ with their corresponding
$k$-diagrams.

If $d_1$, $d_2$ are $k$-diagrams, their \emph{composite configuration}
$\Gamma(d_1, d_2)$ is the graph obtained by placing $d_1$ above $d_2$
and identifying the corresponding vertices in the middle row.  Let
$d_1 \circ d_2$ be the diagram obtained by retaining the edges with
endpoints in the union of the top and bottom rows of vertices in the
composite configuration, along with those vertices, and discarding the
rest of the configuration. The multiplication $\circ$ is associative,
so $(\Ptn_k, \circ)$ is a monoid, the \emph{partition monoid}.

One often identifies diagrams with morphisms in a suitable
category. In this paper, the reader should think of a diagram as
depicting a morphism going from its bottom to top row, so that
products depict compositions in which morphisms act on the left of
arguments.

\subsection{The partition algebra}\noindent
We refer the reader to \cite{HR:05} for basic properties of partition
algebras.  For any $\delta \in \Bbbk$, the \emph{partition algebra}
$\Ptn_k(\delta)$ is a twisted semigroup algebra on $\Ptn_k$. As a
$\Bbbk$-module, $\Ptn_k(\delta) = \Bbbk[\Ptn_k]$, the collection of
$\Bbbk$-linear combinations of elements of $\Ptn_k$. Given
$k$-diagrams $d_1, d_2$,
\begin{equation}\label{e:mult}
  d_1 d_2 = \delta^{N(d_1,d_2)}\, d_3 = \delta^{N(d_1,d_2)}\,(d_1 \circ d_2)
\end{equation}
where $N(d_1,d_2)$ is the number of interior blocks (connecting
vertices in the middle row) in $\Gamma(d_1,d_2)$ that get discarded in
forming $d_3 = d_1 \circ d_2$. Extending this multiplication rule to
linear combinations bilinearly as usual, the set $\Ptn_k(\delta)$
becomes an associative algebra with unit.

Given diagrams $d_1 \in \Ptn_k$, $d_2 \in \Ptn_l$, let $d_1
\otimes d_2 \in \Ptn_{k+l}$ be the diagram obtained by placing $d_1$
to the left of $d_2$. (The notation $\otimes$ in this context is not a
tensor product.)  The following basic diagrams
\begin{equation}
  1 : = \shade{\onepic}\;, \quad p := \shade{\ppic}\;,
  \quad s := \shade{\spic}\;, \quad b:= \shade{\bpic}
\end{equation}
are fundamental building blocks for all partition diagrams. Notice
that $1_k = 1 \otimes \cdots \otimes 1$ (with $k$ factors) is the
identity element of $\Ptn_k(\delta)$; in the sequel we will often
abuse notation by writing $1$ in place of $1_k$. In the set $\Ptn_k$,
define
\begin{align*}
  p_j := 1_{j-1} \otimes p \otimes 1_{k-j} &= \;\shade{\onepic \cdots \onepic
  \quad \ppic \quad\onepic \cdots \onepic}\\
  \intertext{with the isolated vertices in the $j$th column,
    for $j = 1, \dots, k$, and define}
  s_i := 1_{i-1} \otimes s \otimes 1_{k-1-i} &= \;\shade{\onepic \cdots
    \onepic\quad \spic \quad\onepic \cdots \onepic}\\
  b_i := 1_{i-1} \otimes b \otimes 1_{k-1-i} &= \;\shade{\onepic \cdots
    \onepic\quad \bpic \quad\onepic \cdots \onepic}
\end{align*}
for $i = 1, \dots, k-1$. The elements $p_j, s_i, b_i$ form a set of
generators of $\Ptn_k(\delta)$; its defining relations are given in
\cite{HR:05}*{Thm.~1.11}. (Note that $b_i$ is denoted by
$p_{i+\frac{1}{2}}$ in that reference.)
We also need the diagrams
\[
e := \; \shade{\epic} \quad\text{and}\quad e_i:= 1_{i-1} \otimes e \otimes
1_{k-1-i} = \;\shade{\onepic \cdots \onepic\quad \epic \quad\onepic \cdots
\onepic}
\]
for any $i = 1, \dots, k-1$. The reader may easily check that the
elements $e_i$ satisfy the identity
\begin{equation}\label{e:bppb}
  e_i = b_i p_i p_{i+1} b_i 
\end{equation}
in $\Ptn_k(\delta)$; this identity provides an alternative definition of
$e_i$.

\subsection{The partial Brauer algebra}
The subalgebra of $\Ptn_k(\delta)$ spanned by the $k$-diagrams in
which each block has cardinality at most $2$ is the (one-parameter)
\emph{partial Brauer algebra} $\PB_k(\delta,\delta)$ of
\cites{MM,HdM}, who studied its more general two-parameter variant
$\PB_k(\delta,\delta')$, with multiplication defined by
\begin{equation}\label{e:2mult}
  d_1 d_2 = \delta^{N_1(d_1,d_2)} \delta'^{\,N_2(d_1,d_2)}\, d_3 =
  \delta^{N_1(d_1,d_2)} \delta'^{\,N_2(d_1,d_2)}\, (d_1 \circ d_2)
\end{equation}
where $N_1(d_1,d_2)$ (resp., $N_2(d_1,d_2)$) is the number of interior
loops (resp., interior paths, including paths consisting of a single
vertex) in the middle row of $\Gamma(d_1,d_2)$, and $d_3 = d_1 \circ
d_2$ is the product in the partition monoid $\Ptn_k$.

Note that $e_i$, $s_i$ ($i \in [k-1]$) and $p_j$ ($j \in [k]$) belong
to $\PB_k(\delta,\delta')$; in fact, this is a set of generators of
that algebra.  Both $e_i$ and $p_i$ are pseudo-idempotents, satisfying
\begin{equation}
  e_i^2 = \delta\, e_i, \qquad p_i^2 = \delta'\, p_i .
\end{equation}
For any $\delta' \ne 0$, we have $\PB_k(\delta,\delta') \cong
\PB_k(\delta,1)$; see \eqref{e:1-iso} in \S\ref{s:alt}.  Thus it makes
sense to focus on $\PB_k(\delta,1)$.

\subsection{The Motzkin algebra}
A $k$-diagram that can be drawn without any intersections is said to
\emph{planar}.  The \emph{Motzkin algebra} $\M_k(\delta,\delta')$
studied in \cite{BH} is the subalgebra of $\PB_k(\delta,\delta')$
spanned by the planar partial Brauer $k$-diagrams; that is, planar
diagrams with blocks of cardinality at most $2$.

The algebra $\PB_k(\delta,\delta')$ contains elements $r_i :=
p_is_i$ and $l_i := s_i p_i$, depicted by
\begin{align*}
r_i = 1_{i-1} \otimes r \otimes 1_{k-1-i} &= \;\shade{\onepic \cdots
\onepic\quad \rpic \quad\onepic \cdots \onepic} \\
l_i = 1_{i-1} \otimes l \otimes 1_{k-1-i} &= \;\shade{\onepic \cdots
\onepic\quad \lpic \quad\onepic \cdots \onepic}
\end{align*}
for $i \in [k-1]$, where $r, l$ in $\Ptn_2$ are given by
\[
r = \;\shade{\rpic}\;, \qquad l = \;\shade{\lpic} \;.
\]
As the $r_i$, $l_i$ are planar, they belong to the Motzkin algebra
$\M_k(\delta,\delta')$.  The diagrams $p_i$ are also planar, hence
belong to $\M_k(\delta,\delta')$. They satisfy
\begin{equation}
  p_i = r_i l_i = l_{i-1}r_{i-1}
\end{equation}
for all values of the indices for which the equalities are sensible.
It is shown in \cite{BH} that
\begin{equation}
\text{$\M_k(\delta,\delta')$ is generated by the $e_i$, $r_i$, $l_i$}
\quad (i \in [k-1]).
\end{equation}
(The element $e_i$ is denoted by $t_i$ in \cite{BH}.)  A set of
defining relations for these generators can be found in
\cites{HLP}.

\subsection{The Temperley--Lieb algebra}\label{ss:TL}
The Brauer algebra $\B_k(\delta)$ is the subalgebra of
$\Ptn_k(\delta)$ spanned by the $k$-diagrams in which all blocks have
exactly two elements. The Temperley--Lieb algebra $\TL_k(\delta)$ is
the subalgebra of $\B_k(\delta)$ spanned by the planar $k$-diagrams
which are also in the Brauer algebra. It is the subalgebra generated
by $e_1, \dots, e_{k-1}$. As such, it is isomorphic to the algebra
defined by those generators and satisfying the relations
\begin{equation}\label{e:TL-rels}
  e_i^2 = \delta e_i, \quad e_i e_{i\pm 1} e_i = e_i, \quad e_ie_j =
  e_j e_i \text{ if } |i-j|>1.
\end{equation}
The rank of $\TL_k(\delta)$ over $\Bbbk$ is equal to the $k$th Catalan
number $\mathcal{C}_k = \frac{1}{k+1}\binom{2k}{k}$.

It is noteworthy that the algebra morphism defined on generators by
$e_i \mapsto -e_i$ for all $i$ defines an isomorphism of algebras
\begin{equation}\label{e:TL-sign}
  \TL_k(\delta) \cong \TL_k(-\delta).
\end{equation}
So the choice of sign of the parameter is purely a matter of
convenience.

If $\Bbbk$ is a field and $0 \ne q$ is an element of $\Bbbk$ such
that $q^2 \ne 1$, it is well known that
\begin{equation}\label{e:TL-endo}
  \TL_k(\pm(q+q^{-1})) \cong \End_{\UU}(V(1)^{\otimes k})
\end{equation}
where $V(1)$ is the $2$-dimensional ``natural'' representation of the
quantized enveloping algebra $\UU = \UU_q(\fraksl_2)$. There is a
(unique) copy of the trivial $\UU$-module in $V(1) \otimes V(1)$.  In
the isomorphism~\eqref{e:TL-endo}, we identify
\[
e_i = \pm(q+q^{-1})\, 1^{\otimes(i-1)} \otimes \pi \otimes 1^{\otimes
  (k-i-1)}
\]
where $\pi$ is an orthogonal projection onto that trivial module and
$1$ denotes the identity map. See Section~\ref{s:VV} for details.

Finally, if $\Bbbk$ is a field and $q$ a nonzero element of $\Bbbk$
satisfying the condition $\Brackets{1}_q \Brackets{2}_q \cdots
\Brackets{k}_q \ne 0$, where $\Brackets{k}_q$ is the balanced form of
a quantum integer, then $\TL_k(\pm(q+q^{-1}))$ is semisimple over
$\Bbbk$; see Appendix~\ref{a:TL}. In particular, this semisimplicity
statement holds whenever $q$ is not a root of unity.

\section{Alternating bases}\label{s:alt}\noindent
Vaughan Jones \cite{Jones:94} (see also \cites{HR:05,BH1,BH2})
introduced the \emph{orbit} basis $\{o_d \mid d \in \Ptn_k \}$ of
$\Ptn_k(\delta)$, defined as follows.  Given $k$-diagrams $d_1$, $d_2$
write
\begin{equation}\label{e:order}
\text{$d_1 \le d_2 \iff$ each block of $d_1$ is contained in some block
  of $d_2$.}
\end{equation}
The relation $\le$ is a partial order on the set $\Ptn_k$ of
$k$-diagrams. The orbit basis $\{o_d \mid d \in \Ptn_k \}$ is defined
by demanding that the unitriangular relation $d = \sum_{d \le d'}
o_{d'}$ hold for every $k$-diagram $d$.

A different basis $\{\bar{d} \mid d \in \Ptn_k\}$, also in a unitriangular
relation with the diagram basis, is defined by setting
\begin{equation}\label{e:bar}
  \bar{d} = \textstyle \sum_{d' \le d} \; (-1)^{\beta(d)-\beta(d')} \,
  d^{\,\prime}
\end{equation}
for any $k$-diagram $d$, where $\beta(d)$ is the number of blocks of
$d$. This basis is the \emph{alternating} basis; it plays a crucial
role in this paper.

The blocks of a partial Brauer diagram all have cardinality at most
$2$, so blocks are either singletons (isolated vertices) or edges
(having two vertices as endpoints). For partial Brauer diagrams
$d$, $d'$ the relation $d \le d'$ defined in \eqref{e:order} holds if
and only if all the edges of $d$ are also edges of $d'$.
Equivalently, $d \le d'$ if and only if $d$ is obtainable from $d'$ by
excising zero or more of its edges.

If $d$ is a partial Brauer diagram then any term in the right hand
side of the expansion \eqref{e:bar} is (up to sign) also a partial
Brauer diagram. The same holds for planar partial Brauer diagrams
(that is, Motzkin diagrams).  Hence, $\bar{d} \in
\PB_k(\delta,\delta')$ for any partial Brauer diagram $d$, and
similarly $\bar{d} \in \M_k(\delta,\delta')$ for any planar partial
Brauer diagram $d$.

\begin{lem}\label{l:bases}
For any unital commutative ring $\Bbbk$ and any $\delta$, $\delta' \in
\Bbbk$, the sets
\[
\{\bar{d} \mid d \text{ a partial Brauer diagram}\},\quad 
\{\bar{d} \mid d \text{ a Motzkin diagram}\}
\]
are bases of $\PB_k(\delta,\delta')$, $\M_k(\delta,\delta')$
respectively.
\end{lem}

\begin{proof}
By the remarks preceding the lemma, the transition matrix expressing
the $\bar{d}$ in terms of the diagram basis in each algebra is 
unitriangular with respect to any linear order extending $\le$.
\end{proof}

For any invertible $\delta' \in \Bbbk$, by \cites{MM,DG:twin} there is
an algebra isomorphism
\begin{equation}\label{e:1-iso}
  \PB_k(\delta,\delta') \cong \PB_k(\delta,1) 
\end{equation}
defined by replacing the generator $p_i$ by $p_i/\delta'$.  Thus,
there is no loss of generality in setting $\delta' = 1$, so from now
on we work in $\PB_k(\delta,1)$ and in its subalgebra
$\M_k(\delta,1)$.  This is convenient because $p_i$, $1-p_i$ become a
pair of commuting orthogonal idempotents.

For a given partial Brauer diagram $d$, consider the subsets
$\bott{d}$, $\topp{d}$ of $[k]$ (the \emph{bottom}, \emph{top frame},
respectively, of $d$) defined by
\begin{align*}
\bott{d} &= \{i \in [k] \mid \text{vertex $i'$ is non-isolated in $d$} \},\\
\topp{d} &= \{i \in [k] \mid \text{vertex $i$ is non-isolated in $d$} \}.
\end{align*}
Elements of the set $\bott{d}' \cup \topp{d}$ form the \emph{frame} of
$d$ and label the endpoints of the edges in $d$; its complement in
$[k]' \cup [k]$ labels the isolated vertices in~$d$.

\begin{prop}\label{p:prod-form}
Let $d$ be a partial Brauer diagram. The identities:
\begin{enumerate}
\item $\bar{d} = \prod_{i \in \topp{d}} (1-p_i) \; d \prod_{i' \in
  \bott{d}} (1-p_i)$
\item $\bar{d} = \rho_0(d) - \rho_1(d) + \rho_2(d) - \rho_3(d) + \cdots$
\end{enumerate}
hold in $\PB_k(\delta,1)$, where $\rho_i(d)$ is the sum of all
diagrams obtained from $d$ by removing exactly $i$ of its edges. If
$d$ is planar (i.e., a Motzkin diagram) then all terms on the right
hand side of the identities are also planar, hence belong to
$\M_k(\delta,1)$.
\end{prop}

\begin{proof}
(a) It follows from the definition \eqref{e:2mult} of diagram
  multiplication that left (resp., right) multiplication by $p_i$
  removes the edge with endpoint $i$ (resp., $i'$), for any $i \in
  \topp{d}$, $i' \in \bott{d}$.  Expanding the products on the left
  and right of identity (a) thus gives
\[
\bar{d} = \textstyle \sum_{d' \le d} (-1)^{\text{edges}(d)-\text{edges}(d')} d'
\]
where $\text{edges}(d)$ is the number of edges in $d$.  Horizontal
edges are seemingly removed twice, once for each endpoint, but in fact
the second multiplication by the appropriate $p_i$ acts as the
identity, thus doesn't matter.  For partial Brauer diagrams, the above
expansion coincides with the expression in \eqref{e:bar}.

(b) This follows from the displayed equation in the proof of (a).
\end{proof}

\begin{rmk}
The formula in part (b) of the proposition says that, for partial
Brauer diagrams $d$, the element $\bar{d}$ is obtained from $d$ by
\emph{inclusion-exclusion edge removal}. As $\rho_0(d)=d$, the leading
term in the expansion is $d$ itself.
\end{rmk}

An edge of a partial Brauer diagram $d$ is a \emph{cup} (resp.,
\emph{cap}) if both of its endpoints are in $[k]$ (resp., in $[k]'$).
Such edges are also called \emph{horizontal}. For a partial Brauer
diagram $d$, we define
\begin{equation}\label{e:tilde-hat}
\begin{aligned}
\tilde{d} &:=  \prod_{i \in \topp{d}_H} (1-p_i) \; d \prod_{i' \in
  \bott{d}_H} (1-p_i) \\
\hat{d} &:=  \prod_{i \in \topp{d}_V} (1-p_i) \; d \prod_{i' \in
  \bott{d}_V} (1-p_i)
\end{aligned}
\end{equation}
where $\topp{d}_H$, $\bott{d}_H$ respectively index the endpoints of
horizontal edges in the top, bottom rows of $d$, and similarly
$\topp{d}_V$, $\bott{d}_V$ respectively index the endpoints of
vertical edges in the top, bottom rows of $d$. Notice that either
product in the definition of $\hat{d}$ may be omitted without changing
the result.  The element $\tilde{d}$ (resp., $\hat{d}$) is obtained
from $d$ by \emph{inclusion-exclusion horizontal (resp., vertical)
edge removal}.  We linearly extend the notations $\bar{d}$,
$\tilde{d}$, $\hat{d}$ to linear combinations of diagrams.

\begin{prop}\label{prop:tilde-hat}
Let $d$ be a partial Brauer diagram. Then
\begin{enumerate}
\item $\bar{d} = \hat{\tilde{d}} = \tilde{\hat{d}}$
\item $\tilde{d} = \rho^H_0(d) - \rho^H_1(d) + \rho^H_2(d) - \cdots$
\item $\hat{d} = \rho^V_0(d) - \rho^V_1(d) + \rho^V_2(d) - \cdots$
\end{enumerate}
where $\rho^H_i(d)$ (resp., $\rho^V_i(d)$) is the sum of all diagrams
obtained from $d$ by removing $i$ of its horizontal (resp., vertical)
edges. Hence the sets
\[
\{\tilde{d} \mid d \in \PB_k\}, \quad \{\tilde{d} \mid d \in \M_k\}
\]
are $\Bbbk$-bases of $\PB_k(\delta,1)$, $\M_k(\delta,1)$, respectively.
\end{prop}

\begin{proof}
(a) follows from the product formula in Proposition
  \ref{p:prod-form}(a) and the fact that the $p_i$ pairwise commute.

(b), (c) follow by expanding the right hand side in \eqref{e:tilde-hat}.

Applying the operator $d\mapsto \tilde{d}$ to both sides of the
identity in part (c) shows that if $d$ is a partial Brauer diagram
then $\bar{d}$ is expressible as an alternating sum of the form
\[
\bar{d} = \tilde{d} - \tilde{d}_1 + \tilde{d}_2 - \cdots
\]
where each $d_i < d$. So the transition between the sets $\{\bar{d}
\mid d \in \PB_k \}$ and $\{\tilde{d} \mid d \in \PB_k \}$ is
unitriangular. The same holds if we restrict to Motzkin diagrams. The
final claim now follows from Lemma \ref{l:bases}.
\end{proof}

We will need to consider various subalgebras of the Motzkin algebra
$\M_k(\delta)$. Let $\RR_k$ be the subalgebra generated by $r_1,
\dots, r_{k-1}$ and $\LL_k$ the subalgebra generated by $l_1, \dots,
l_{k-1}$. Write $\RL_k$ for the subalgebra generated by $\RR_k$ and
$\LL_k$. Notice that $\tilde{d}=d$ for any $d$ in $\RL_k$, since such
$d$ have no horizontal edges.

\begin{thm}\label{t:bar-mult}
Suppose that $d_1$, $d_2$ are partial Brauer $k$-diagrams.  Let
$N_1(d_1,d_2)$ be the number of closed loops in the middle of
$\Gamma(d_1,d_2)$, as in equation~\eqref{e:2mult}, and set $d_3 = d_1
\circ d_2$, the product in the partition monoid~$\Ptn_k$.  Let
\[
\Omega(d_1,d_2) = \big( ([k]\setminus \bott{d_1}) \cap \topp{d_2}_H
\big) \cup \big( ([k]\setminus \topp{d_2}) \cap \bott{d_1}_H \big),
\]
the set indexing the vertices in the middle row of $\Gamma(d_1,d_2)$
for which an isolated vertex in one diagram is identified with a
horizontal edge endpoint in the other.  Then:
\begin{enumerate}
\item $\bar{d}_1 \, \bar{d}_2 =
  \begin{cases}
    (\delta-1)^{N_1(d_1,d_2)} \, \bar{d}_3 & \text{if $\bott{d_1} =
      \topp{d_2}$}\\
     \phantom{(\delta-1)} 0 & \text{otherwise}.
  \end{cases}
  $
  
\item $\tilde{d}_1 \, \tilde{d}_2 =
  \begin{cases}
  (\delta-1)^{N_1(d_1,d_2)} \, \prod_{i\in S} (1-p_i) \,\tilde{d}_3 &
    \text{if $\Omega(d_1,d_2)$ is empty}\\ \phantom{(\delta-1)} 0 &
    \text{otherwise}.
  \end{cases}
  $
\end{enumerate}
The set $S$ in formula \textup{(b)} is the set of indices on the top
vertex of any through edge that snakes through some cups and caps in
the middle row of $\Gamma(d_1,d_2)$ before emerging to connect to a
vertex in the bottom row.
\end{thm}

\begin{proof}
(a) The proof is based on the formula in Proposition \ref{p:prod-form}(a).
First suppose that $\bott{d_1}$ does not match $\topp{d_2}$. Then
there must be at least one isolated vertex that matches up with the
endpoint of some edge. We can always insert a copy of $p_i$
corresponding to that vertex, as multiplication by $p_i$ is identity
on an isolated vertex in the $i$th position. This shows that the
product $\bar{d}_1\; \bar{d}_2 = 0$.

From now on, suppose that $\bott{d_1}$ matches $\topp{d_2}$.  There
are four cases to consider. First, suppose that two propagating edges
meet in the middle row of $\Gamma(d_1,d_2)$ at the $i$th identified
vertex. Then the idempotent $1-p_i$ can be commuted to both sides of
$d_1  d_2$, as $1-p_j$ on the left and $1-p_m$ on the right,
where the corresponding edge in $d_3 = d_1d_2$ connects the $j$th
vertex on the top row to the $m$th on the bottom.

Now suppose that there is an added cup in $d_3$ that is not
present in $d_1$. This means two propagating edges in $d_1$ join up
with a connected path in the middle of $\Gamma(d_1,d_2)$ to form that
additional cup. In this situation, we can commute all the interior
idempotents on the path up to the top (to the left of $d_3$).

The case of an added cap in $d_3$ that is not present in $d_2$ is
analogous to the previous case.

It remains only to consider loops in the middle row of
$\Gamma(d_1,d_2)$. For each such loop, the product of the middle
idempotents is equivalent to multiplication by the scalar
$\delta-1$. This completes the proof of (a).

(b) The proof of part (b) is similar to that for part (a), based on
the formula in Proposition~\ref{prop:tilde-hat}(b).
\end{proof}

For example, if $k=3$ one may verify that $\tilde{e}_1 \tilde{e}_2$ is
equal to the $8$-term linear combination
\[
\begin{aligned} \tilde{e}_1 \tilde{e}_2 = \;
\shade{ \begin{tikzpicture}[scale = 0.35,thick, baseline={(0,-1ex/2)}] 
\tikzstyle{vertex} = [shape = circle, minimum size = 4pt, inner sep = 1pt] 
\node[vertex] (G--3) at (3.0, -1) [shape = circle, draw,fill=black] {}; 
\node[vertex] (G--2) at (1.5, -1) [shape = circle, draw,fill=black] {}; 
\node[vertex] (G--1) at (0.0, -1) [shape = circle, draw,fill=black] {}; 
\node[vertex] (G-3) at (3.0, 1) [shape = circle, draw,fill=black] {}; 
\node[vertex] (G-1) at (0.0, 1) [shape = circle, draw,fill=black] {}; 
\node[vertex] (G-2) at (1.5, 1) [shape = circle, draw,fill=black] {}; 
\draw[] (G--3) .. controls +(-0.5, 0.5) and +(0.5, 0.5) .. (G--2); 
\draw[] (G-3) .. controls +(-1, -1) and +(1, 1) .. (G--1); 
\draw[] (G-1) .. controls +(0.5, -0.5) and +(-0.5, -0.5) .. (G-2); 
\end{tikzpicture}}
\;&-\; \shade{\begin{tikzpicture}[scale = 0.35,thick, baseline={(0,-1ex/2)}] 
\tikzstyle{vertex} = [shape = circle, minimum size = 4pt, inner sep = 1pt] 
\node[vertex] (G--3) at (3.0, -1) [shape = circle, draw,fill=black] {}; 
\node[vertex] (G--2) at (1.5, -1) [shape = circle, draw,fill=black] {}; 
\node[vertex] (G--1) at (0.0, -1) [shape = circle, draw,fill=black] {}; 
\node[vertex] (G-3) at (3.0, 1) [shape = circle, draw,fill=black] {}; 
\node[vertex] (G-1) at (0.0, 1) [shape = circle, draw,fill=black] {}; 
\node[vertex] (G-2) at (1.5, 1) [shape = circle, draw,fill=black] {}; 
\draw[] (G--3) .. controls +(-0.5, 0.5) and +(0.5, 0.5) .. (G--2); 
\draw[] (G-3) .. controls +(-1, -1) and +(1, 1) .. (G--1); 
\end{tikzpicture}}
\;-\; \shade{\begin{tikzpicture}[scale = 0.35,thick, baseline={(0,-1ex/2)}] 
\tikzstyle{vertex} = [shape = circle, minimum size = 4pt, inner sep = 1pt] 
\node[vertex] (G--3) at (3.0, -1) [shape = circle, draw,fill=black] {}; 
\node[vertex] (G--2) at (1.5, -1) [shape = circle, draw,fill=black] {}; 
\node[vertex] (G--1) at (0.0, -1) [shape = circle, draw,fill=black] {}; 
\node[vertex] (G-3) at (3.0, 1) [shape = circle, draw,fill=black] {}; 
\node[vertex] (G-1) at (0.0, 1) [shape = circle, draw,fill=black] {}; 
\node[vertex] (G-2) at (1.5, 1) [shape = circle, draw,fill=black] {}; 
\draw[] (G-3) .. controls +(-1, -1) and +(1, 1) .. (G--1); 
\draw[] (G-1) .. controls +(0.5, -0.5) and +(-0.5, -0.5) .. (G-2); 
\end{tikzpicture}}
\;+\; \shade{\begin{tikzpicture}[scale = 0.35,thick, baseline={(0,-1ex/2)}] 
\tikzstyle{vertex} = [shape = circle, minimum size = 4pt, inner sep = 1pt] 
\node[vertex] (G--3) at (3.0, -1) [shape = circle, draw,fill=black] {}; 
\node[vertex] (G--2) at (1.5, -1) [shape = circle, draw,fill=black] {}; 
\node[vertex] (G--1) at (0.0, -1) [shape = circle, draw,fill=black] {}; 
\node[vertex] (G-3) at (3.0, 1) [shape = circle, draw,fill=black] {}; 
\node[vertex] (G-1) at (0.0, 1) [shape = circle, draw,fill=black] {}; 
\node[vertex] (G-2) at (1.5, 1) [shape = circle, draw,fill=black] {}; 
\draw[] (G-3) .. controls +(-1, -1) and +(1, 1) .. (G--1); 
\end{tikzpicture}} \quad \\[1.0ex] \quad
\;&-\; \shade{\begin{tikzpicture}[scale = 0.35,thick, baseline={(0,-1ex/2)}] 
\tikzstyle{vertex} = [shape = circle, minimum size = 4pt, inner sep = 1pt] 
\node[vertex] (G--3) at (3.0, -1) [shape = circle, draw,fill=black] {}; 
\node[vertex] (G--2) at (1.5, -1) [shape = circle, draw,fill=black] {}; 
\node[vertex] (G--1) at (0.0, -1) [shape = circle, draw,fill=black] {}; 
\node[vertex] (G-1) at (0.0, 1) [shape = circle, draw,fill=black] {}; 
\node[vertex] (G-2) at (1.5, 1) [shape = circle, draw,fill=black] {}; 
\node[vertex] (G-3) at (3.0, 1) [shape = circle, draw,fill=black] {}; 
\draw[] (G--3) .. controls +(-0.5, 0.5) and +(0.5, 0.5) .. (G--2); 
\draw[] (G-1) .. controls +(0.5, -0.5) and +(-0.5, -0.5) .. (G-2); 
\end{tikzpicture}}
\;+\; \shade{\begin{tikzpicture}[scale = 0.35,thick, baseline={(0,-1ex/2)}] 
\tikzstyle{vertex} = [shape = circle, minimum size = 4pt, inner sep = 1pt] 
\node[vertex] (G--3) at (3.0, -1) [shape = circle, draw,fill=black] {}; 
\node[vertex] (G--2) at (1.5, -1) [shape = circle, draw,fill=black] {}; 
\node[vertex] (G--1) at (0.0, -1) [shape = circle, draw,fill=black] {}; 
\node[vertex] (G-1) at (0.0, 1) [shape = circle, draw,fill=black] {}; 
\node[vertex] (G-2) at (1.5, 1) [shape = circle, draw,fill=black] {}; 
\node[vertex] (G-3) at (3.0, 1) [shape = circle, draw,fill=black] {}; 
\draw[] (G--3) .. controls +(-0.5, 0.5) and +(0.5, 0.5) .. (G--2); 
\end{tikzpicture}}
\;+\; \shade{\begin{tikzpicture}[scale = 0.35,thick, baseline={(0,-1ex/2)}] 
\tikzstyle{vertex} = [shape = circle, minimum size = 4pt, inner sep = 1pt] 
\node[vertex] (G--3) at (3.0, -1) [shape = circle, draw,fill=black] {}; 
\node[vertex] (G--2) at (1.5, -1) [shape = circle, draw,fill=black] {}; 
\node[vertex] (G--1) at (0.0, -1) [shape = circle, draw,fill=black] {}; 
\node[vertex] (G-1) at (0.0, 1) [shape = circle, draw,fill=black] {}; 
\node[vertex] (G-2) at (1.5, 1) [shape = circle, draw,fill=black] {}; 
\node[vertex] (G-3) at (3.0, 1) [shape = circle, draw,fill=black] {}; 
\draw[] (G-1) .. controls +(0.5, -0.5) and +(-0.5, -0.5) .. (G-2); 
\end{tikzpicture}}
\;-\; \shade{\begin{tikzpicture}[scale = 0.35,thick, baseline={(0,-1ex/2)}] 
\tikzstyle{vertex} = [shape = circle, minimum size = 4pt, inner sep = 1pt] 
\node[vertex] (G--3) at (3.0, -1) [shape = circle, draw,fill=black] {}; 
\node[vertex] (G--2) at (1.5, -1) [shape = circle, draw,fill=black] {}; 
\node[vertex] (G--1) at (0.0, -1) [shape = circle, draw,fill=black] {}; 
\node[vertex] (G-1) at (0.0, 1) [shape = circle, draw,fill=black] {}; 
\node[vertex] (G-2) at (1.5, 1) [shape = circle, draw,fill=black] {}; 
\node[vertex] (G-3) at (3.0, 1) [shape = circle, draw,fill=black] {}; 
\end{tikzpicture}} 
\end{aligned} 
\]
illustrating the equality $\tilde{e}_1 \tilde{e}_2 = (1-p_3)\,
\widetilde{e_1e_2}$ in Theorem~\ref{t:bar-mult}(b).

\begin{cor}\label{c:tilde-cor}
If $x$ is a diagram in $\RL_k$ and $d$ is a partial Brauer diagram
then
\begin{enumerate}
\item $x\, \tilde{d} =
  \begin{cases}
    \widetilde{xd} & \text{if $\topp{d}_H \subset \bott{x}$}\\
    \; 0 &\text{otherwise}.
  \end{cases}
  $
\item $\tilde{d}\,x =
  \begin{cases}
    \widetilde{dx} & \text{if $\bott{d}_H \subset \topp{x}$}\\
    \;0 & \text{otherwise}.
  \end{cases}
  $
\end{enumerate}
\end{cor}

\begin{proof}
These formulas follow easily from part (b) of
Theorem~\ref{t:bar-mult}.  We prove formula (a). Since $x$ belongs to
$\RL_k$, it has no horizontal edges, so $\bott{x}_H$ is empty. This
means that $\Omega(x,d)$ is empty if and only if $\topp{d}_H \subset
\bott{x}$, and furthermore, there are no interior closed loops in
$\Gamma(x,d)$. Formula (a) thus follows. The proof of formula (b) is
symmetric.
\end{proof}

\section{The partial Temperley--Lieb algebra}\label{s:PTL}\noindent
We are now ready to define the main object of study for this paper.
Henceforth we set $\M_k(\delta) := \M_k(\delta,1)$.

We say that a Motzkin diagram is \emph{balanced} if it has the same
number of cups as caps. Equivalently, a Motzkin diagram is balanced if
the number of isolated vertices in each of its rows is the same.  Let
\begin{equation*}
  \D(k) := \{d \mid d \text{ is a balanced Motzkin $k$-diagram} \}.
\end{equation*}
We define the \emph{partial Temperley--Lieb algebra} $\PTL_k(\delta)$
to be the linear span of $\{ \tilde{d} \mid d \in \D(k) \}$. That
$\PTL_k(\delta)$ is a subalgebra of $\M_k(\delta)$ follows from
Theorem~\ref{t:bar-mult}.

\begin{thm}\label{t:bases}
Let $\Bbbk$ be a commutative ring and fix $\delta \in \Bbbk$. Then
either of the sets $\{\tilde{d} \mid d \in \D(k)\}$, $\{\bar{d} \mid d
\in \D(k)\}$ is a $\Bbbk$-basis of $\PTL_k(\delta)$.
\end{thm}

\begin{proof}
The set $\{\bar{d} \mid d \in \D(k)\}$ is linearly independent by
Lemma~\ref{l:bases}. Since it spans the algebra, it is a
basis. Furthermore, as in the final paragraph of the proof of
Proposition~\ref{prop:tilde-hat}, the transition matrices between the
sets $\{\tilde{d} \mid d \in \D(k)\}$, $\{\bar{d} \mid d \in \D(k)\}$
are unitriangular, so $\{\tilde{d} \mid d \in \D(k)\}$ is also a
basis.
\end{proof}

Let $\D_n(k)$ be the subset of $\D(k)$ consisting of those balanced
Motzkin $k$-diagrams having exactly $n$ edges. Notice that $\D_k(k)$
is the set of Temperley--Lieb diagrams on $2k$ vertices.  To any $d
\in \D_n(k)$, we associate a triple $(A,t,B)$, where $t \in \D_n(n)$
is the unique Temperley--Lieb diagram on $2n$ vertices obtained by
deleting the isolated vertices in $d$, and $A$, $B$ are the subsets of
$[k]$ respectively indexing the non-isolated vertices in the top,
bottom row of $d$. For example,
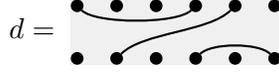
\[
d = \;\shade{%
\begin{tikzpicture}[scale = 0.35,thick, baseline={(0,-1ex/2)}] 
\tikzstyle{vertex} = [shape = circle, minimum size = 4pt, inner sep = 1pt] 
\node[vertex] (G--6) at (7.5, -1) [shape = circle, draw,fill=black] {}; 
\node[vertex] (G--4) at (4.5, -1) [shape = circle, draw,fill=black] {}; 
\node[vertex] (G--5) at (6.0, -1) [shape = circle, draw,fill=black] {}; 
\node[vertex] (G--3) at (3.0, -1) [shape = circle, draw,fill=black] {}; 
\node[vertex] (G--2) at (1.5, -1) [shape = circle, draw,fill=black] {}; 
\node[vertex] (G-5) at (6.0, 1) [shape = circle, draw,fill=black] {}; 
\node[vertex] (G--1) at (0.0, -1) [shape = circle, draw,fill=black] {}; 
\node[vertex] (G-1) at (0.0, 1) [shape = circle, draw,fill=black] {}; 
\node[vertex] (G-4) at (4.5, 1) [shape = circle, draw,fill=black] {}; 
\node[vertex] (G-2) at (1.5, 1) [shape = circle, draw,fill=black] {}; 
\node[vertex] (G-3) at (3.0, 1) [shape = circle, draw,fill=black] {}; 
\node[vertex] (G-6) at (7.5, 1) [shape = circle, draw,fill=black] {}; 
\draw (G--6) .. controls +(-0.6, 0.6) and +(0.6, 0.6) .. (G--4); 
\draw (G-5) .. controls +(-1, -1) and +(1, 1) .. (G--2); 
\draw (G-1) .. controls +(0.7, -0.7) and +(-0.7, -0.7) .. (G-4); 
\end{tikzpicture}
}
\]
corresponds to the triple $(A,t,B)$ where $A = \{1,4,5\}$, $B =
\{2,4,6\}$, and
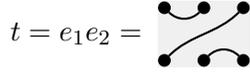
\[
t = e_1e_2 = \;\shade{%
\begin{tikzpicture}[scale = 0.35,thick, baseline={(0,-1ex/2)}] 
\tikzstyle{vertex} = [shape = circle, minimum size = 4pt, inner sep = 1pt] 
\node[vertex] (G--3) at (3.0, -1) [shape = circle, draw,fill=black] {}; 
\node[vertex] (G--2) at (1.5, -1) [shape = circle, draw,fill=black] {}; 
\node[vertex] (G--1) at (0.0, -1) [shape = circle, draw,fill=black] {}; 
\node[vertex] (G-3) at (3.0, 1) [shape = circle, draw,fill=black] {}; 
\node[vertex] (G-1) at (0.0, 1) [shape = circle, draw,fill=black] {}; 
\node[vertex] (G-2) at (1.5, 1) [shape = circle, draw,fill=black] {}; 
\draw (G--3) .. controls +(-0.5, 0.5) and +(0.5, 0.5) .. (G--2); 
\draw (G-3) .. controls +(-1, -1) and +(1, 1) .. (G--1); 
\draw (G-1) .. controls +(0.5, -0.5) and +(-0.5, -0.5) .. (G-2); 
\end{tikzpicture}
}
\]
in $\D_3(3)$. Since any diagram $d$ is reconstructible from its
triple, the following result is clear.

\begin{lem}\label{l:triples}
The map $d \mapsto (A,t,B)$ defines a bijection between $\D_n(k)$ and
the set of triples of the above form.
\end{lem}

From now on, write $d(A,t,B)$ for the $k$-diagram in $\D_n(k)$
corresponding in the above lemma to a given triple $(A,t,B)$ such that
$A$, $B$ are subsets of $[k]$ of cardinality $n$ and $t \in
\D_n(n)$. In other words, the map $(A,t,B) \mapsto d(A,t,B)$ is the
inverse of the bijection in Lemma~\ref{l:triples}.

The cardinality of the set $\TL_n$ of Temperley--Lieb $n$-diagrams is
equal to the $n$th Catalan number $\mathcal{C}_n = \frac{1}{n+1}
\binom{2n}{n}$.  As $\D(k) = \bigcup_{n=0}^k \D_n(k)$ (disjoint union),
by combining Lemma~\ref{l:triples} with Theorem~\ref{t:bases}, we have
\begin{equation}\label{e:PTL-dim}
\rank_\Bbbk \PTL_k(\delta) = |\D(k)| = \sum_{n=0}^k \binom{k}{n}^2
\mathcal{C}_n \, .
\end{equation}

If $t \in \D_n(n)$ is a Temperley--Lieb $n$-diagram, there are two
equally natural ways to extend $t$ to a diagram in $\D_n(k)$, for any
$k \ge n$. Either of the maps
\begin{equation}\label{e:extend}
t \mapsto t \otimes \omega_{k-n}, \qquad t \mapsto t \otimes 1_{k-n}
\end{equation}
will do the job, where $\omega_j$ is the $j$-diagram in which all $2j$
vertices are isolated. The following observation is an immediate
consequence of Theorem~\ref{t:bar-mult}(a).

\begin{lem}\label{l:TL-iso}
Write $t_0 = t \otimes \omega_{k-n}$, $t_1 = t \otimes 1_{k-n}$ for
the image of $t$ in $\D_n(n)$ under the map \eqref{e:extend}. For any
$n \le k$, either of the linear maps such that $t \mapsto \bar{t}_0$,
$t \mapsto \bar{t}_1$ defines an algebra isomorphism of
$\TL_n(\delta-1)$ with a subalgebra of $\M_k(\delta)$.
\end{lem}

Let $\RP_k$ (resp., $\LP_k$) be the subalgebra of $\M_k(\delta)$
generated by all $r_i$, $p_i$ (resp, all $l_i$, $p_i$). The following
result is a variant of \cite{BH}*{(2.12)}.

\begin{lem}\label{l:Motzkin-factors}
  Let $d = d(A,t,B)$ be a diagram in $\D_n(k)$, where $t$ is in
  $\D_n(n)$. In the Motzkin algebra $\M_k(\delta)$, we have the
  factorizations
  \[
  d = r_A (t \otimes \omega_{k-n}) l_B , \qquad d = r_A (t \otimes
  1_{k-n}) l_B
  \]
  where $r_A \in \RP_k$ (resp., $l_B \in \LP_k$) is the unique
  $k$-diagram with edges from the first $n$ bottom-row (resp.,
  top-row) vertices connecting to the top-row (resp., top-row)
  vertices indexed by $A$ (resp., $B$), in order.
\end{lem}

\begin{proof}
The second factorization is the $RTL$ factorization in
\cite{BH}*{\S2}. It doesn't matter in that argument if we replace the
identity edges in $1_{k-n}$ by the isolated vertices in
$\omega_{k-n}$, which yields the first factorization.
\end{proof}

We illustrate the proof of Lemma~\ref{l:Motzkin-factors}.  As an
example of the $RTL$ factorization in \cite{BH}, we have
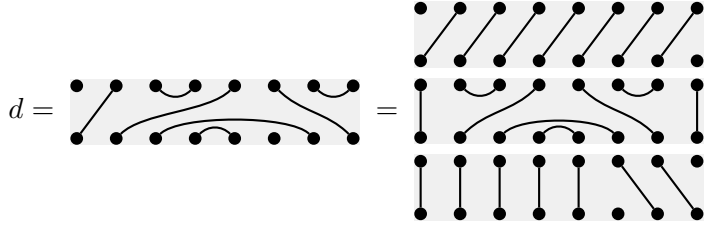
\[
d = \;
\shade{\begin{tikzpicture}[scale = 0.35,thick, baseline={(0,-1ex/2)}] 
\tikzstyle{vertex} = [shape = circle, minimum size = 4pt, inner sep = 1pt] 
\node[vertex] (G--8) at (10.5, -1) [shape = circle, draw,fill=black] {}; 
\node[vertex] (G-6) at (7.5, 1) [shape = circle, draw,fill=black] {}; 
\node[vertex] (G--7) at (9.0, -1) [shape = circle, draw,fill=black] {}; 
\node[vertex] (G--3) at (3.0, -1) [shape = circle, draw,fill=black] {}; 
\node[vertex] (G--6) at (7.5, -1) [shape = circle, draw,fill=black] {}; 
\node[vertex] (G--5) at (6.0, -1) [shape = circle, draw,fill=black] {}; 
\node[vertex] (G--4) at (4.5, -1) [shape = circle, draw,fill=black] {}; 
\node[vertex] (G--2) at (1.5, -1) [shape = circle, draw,fill=black] {}; 
\node[vertex] (G-5) at (6.0, 1) [shape = circle, draw,fill=black] {}; 
\node[vertex] (G--1) at (0.0, -1) [shape = circle, draw,fill=black] {}; 
\node[vertex] (G-2) at (1.5, 1) [shape = circle, draw,fill=black] {}; 
\node[vertex] (G-1) at (0.0, 1) [shape = circle, draw,fill=black] {}; 
\node[vertex] (G-3) at (3.0, 1) [shape = circle, draw,fill=black] {}; 
\node[vertex] (G-4) at (4.5, 1) [shape = circle, draw,fill=black] {}; 
\node[vertex] (G-7) at (9.0, 1) [shape = circle, draw,fill=black] {}; 
\node[vertex] (G-8) at (10.5, 1) [shape = circle, draw,fill=black] {}; 
\draw (G-6) .. controls +(1, -1) and +(-1, 1) .. (G--8); 
\draw (G--7) .. controls +(-0.9, 0.9) and +(0.9, 0.9) .. (G--3); 
\draw (G--5) .. controls +(-0.5, 0.5) and +(0.5, 0.5) .. (G--4); 
\draw (G-5) .. controls +(-1, -1) and +(1, 1) .. (G--2); 
\draw (G-2) .. controls +(-0.75, -1) and +(0.75, 1) .. (G--1); 
\draw (G-3) .. controls +(0.5, -0.5) and +(-0.5, -0.5) .. (G-4); 
\draw (G-7) .. controls +(0.5, -0.5) and +(-0.5, -0.5) .. (G-8); 
\end{tikzpicture}
}
\; = \;
\begin{aligned}
\shade{\begin{tikzpicture}[scale = 0.35,thick, baseline={(0,-1ex/2)}] 
\tikzstyle{vertex} = [shape = circle, minimum size = 4pt, inner sep = 1pt] 
\node[vertex] (G--8) at (10.5, -1) [shape = circle, draw,fill=black] {}; 
\node[vertex] (G--7) at (9.0, -1) [shape = circle, draw,fill=black] {}; 
\node[vertex] (G-8) at (10.5, 1) [shape = circle, draw,fill=black] {}; 
\node[vertex] (G--6) at (7.5, -1) [shape = circle, draw,fill=black] {}; 
\node[vertex] (G-7) at (9.0, 1) [shape = circle, draw,fill=black] {}; 
\node[vertex] (G--5) at (6.0, -1) [shape = circle, draw,fill=black] {}; 
\node[vertex] (G-6) at (7.5, 1) [shape = circle, draw,fill=black] {}; 
\node[vertex] (G--4) at (4.5, -1) [shape = circle, draw,fill=black] {}; 
\node[vertex] (G-5) at (6.0, 1) [shape = circle, draw,fill=black] {}; 
\node[vertex] (G--3) at (3.0, -1) [shape = circle, draw,fill=black] {}; 
\node[vertex] (G-4) at (4.5, 1) [shape = circle, draw,fill=black] {}; 
\node[vertex] (G--2) at (1.5, -1) [shape = circle, draw,fill=black] {}; 
\node[vertex] (G-3) at (3.0, 1) [shape = circle, draw,fill=black] {}; 
\node[vertex] (G--1) at (0.0, -1) [shape = circle, draw,fill=black] {}; 
\node[vertex] (G-2) at (1.5, 1) [shape = circle, draw,fill=black] {}; 
\node[vertex] (G-1) at (0.0, 1) [shape = circle, draw,fill=black] {}; 
\draw (G-8) .. controls +(-0.75, -1) and +(0.75, 1) .. (G--7); 
\draw (G-7) .. controls +(-0.75, -1) and +(0.75, 1) .. (G--6); 
\draw (G-6) .. controls +(-0.75, -1) and +(0.75, 1) .. (G--5); 
\draw (G-5) .. controls +(-0.75, -1) and +(0.75, 1) .. (G--4); 
\draw (G-4) .. controls +(-0.75, -1) and +(0.75, 1) .. (G--3); 
\draw (G-3) .. controls +(-0.75, -1) and +(0.75, 1) .. (G--2); 
\draw (G-2) .. controls +(-0.75, -1) and +(0.75, 1) .. (G--1); 
\end{tikzpicture}
} \\
\shade{\begin{tikzpicture}[scale = 0.35,thick, baseline={(0,-1ex/2)}] 
\tikzstyle{vertex} = [shape = circle, minimum size = 4pt, inner sep = 1pt] 
\node[vertex] (G--8) at (10.5, -1) [shape = circle, draw,fill=black] {}; 
\node[vertex] (G-8) at (10.5, 1) [shape = circle, draw,fill=black] {}; 
\node[vertex] (G--7) at (9.0, -1) [shape = circle, draw,fill=black] {}; 
\node[vertex] (G-5) at (6.0, 1) [shape = circle, draw,fill=black] {}; 
\node[vertex] (G--6) at (7.5, -1) [shape = circle, draw,fill=black] {}; 
\node[vertex] (G--3) at (3.0, -1) [shape = circle, draw,fill=black] {}; 
\node[vertex] (G--5) at (6.0, -1) [shape = circle, draw,fill=black] {}; 
\node[vertex] (G--4) at (4.5, -1) [shape = circle, draw,fill=black] {}; 
\node[vertex] (G--2) at (1.5, -1) [shape = circle, draw,fill=black] {}; 
\node[vertex] (G-4) at (4.5, 1) [shape = circle, draw,fill=black] {}; 
\node[vertex] (G--1) at (0.0, -1) [shape = circle, draw,fill=black] {}; 
\node[vertex] (G-1) at (0.0, 1) [shape = circle, draw,fill=black] {}; 
\node[vertex] (G-2) at (1.5, 1) [shape = circle, draw,fill=black] {}; 
\node[vertex] (G-3) at (3.0, 1) [shape = circle, draw,fill=black] {}; 
\node[vertex] (G-6) at (7.5, 1) [shape = circle, draw,fill=black] {}; 
\node[vertex] (G-7) at (9.0, 1) [shape = circle, draw,fill=black] {}; 
\draw (G-8) .. controls +(0, -1) and +(0, 1) .. (G--8); 
\draw (G-5) .. controls +(1, -1) and +(-1, 1) .. (G--7); 
\draw (G--6) .. controls +(-0.8, 0.8) and +(0.8, 0.8) .. (G--3); 
\draw (G--5) .. controls +(-0.5, 0.5) and +(0.5, 0.5) .. (G--4); 
\draw (G-4) .. controls +(-1, -1) and +(1, 1) .. (G--2); 
\draw (G-1) .. controls +(0, -1) and +(0, 1) .. (G--1); 
\draw (G-2) .. controls +(0.5, -0.5) and +(-0.5, -0.5) .. (G-3); 
\draw (G-6) .. controls +(0.5, -0.5) and +(-0.5, -0.5) .. (G-7); 
\end{tikzpicture}
}   \\
\shade{\begin{tikzpicture}[scale = 0.35,thick, baseline={(0,-1ex/2)}] 
\tikzstyle{vertex} = [shape = circle, minimum size = 4pt, inner sep = 1pt] 
\node[vertex] (G--8) at (10.5, -1) [shape = circle, draw,fill=black] {}; 
\node[vertex] (G-7) at (9.0, 1) [shape = circle, draw,fill=black] {}; 
\node[vertex] (G--7) at (9.0, -1) [shape = circle, draw,fill=black] {}; 
\node[vertex] (G-6) at (7.5, 1) [shape = circle, draw,fill=black] {}; 
\node[vertex] (G--6) at (7.5, -1) [shape = circle, draw,fill=black] {}; 
\node[vertex] (G--5) at (6.0, -1) [shape = circle, draw,fill=black] {}; 
\node[vertex] (G-5) at (6.0, 1) [shape = circle, draw,fill=black] {}; 
\node[vertex] (G--4) at (4.5, -1) [shape = circle, draw,fill=black] {}; 
\node[vertex] (G-4) at (4.5, 1) [shape = circle, draw,fill=black] {}; 
\node[vertex] (G--3) at (3.0, -1) [shape = circle, draw,fill=black] {}; 
\node[vertex] (G-3) at (3.0, 1) [shape = circle, draw,fill=black] {}; 
\node[vertex] (G--2) at (1.5, -1) [shape = circle, draw,fill=black] {}; 
\node[vertex] (G-2) at (1.5, 1) [shape = circle, draw,fill=black] {}; 
\node[vertex] (G--1) at (0.0, -1) [shape = circle, draw,fill=black] {}; 
\node[vertex] (G-1) at (0.0, 1) [shape = circle, draw,fill=black] {}; 
\node[vertex] (G-8) at (10.5, 1) [shape = circle, draw,fill=black] {}; 
\draw (G-7) .. controls +(0.75, -1) and +(-0.75, 1) .. (G--8); 
\draw (G-6) .. controls +(0.75, -1) and +(-0.75, 1) .. (G--7); 
\draw (G-5) .. controls +(0, -1) and +(0, 1) .. (G--5); 
\draw (G-4) .. controls +(0, -1) and +(0, 1) .. (G--4); 
\draw (G-3) .. controls +(0, -1) and +(0, 1) .. (G--3); 
\draw (G-2) .. controls +(0, -1) and +(0, 1) .. (G--2); 
\draw (G-1) .. controls +(0, -1) and +(0, 1) .. (G--1); 
\end{tikzpicture}
}  
\end{aligned}
\]
in which the middle diagram is the diagram $t \otimes 1_1$,
where $t \in \D_7(7)$ is the diagram
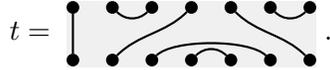
\[
t = \;
\shade{\begin{tikzpicture}[scale = 0.35,thick, baseline={(0,-1ex/2)}] 
\tikzstyle{vertex} = [shape = circle, minimum size = 4pt, inner sep = 1pt] 
\node[vertex] (G--7) at (9.0, -1) [shape = circle, draw,fill=black] {}; 
\node[vertex] (G-5) at (6.0, 1) [shape = circle, draw,fill=black] {}; 
\node[vertex] (G--6) at (7.5, -1) [shape = circle, draw,fill=black] {}; 
\node[vertex] (G--3) at (3.0, -1) [shape = circle, draw,fill=black] {}; 
\node[vertex] (G--5) at (6.0, -1) [shape = circle, draw,fill=black] {}; 
\node[vertex] (G--4) at (4.5, -1) [shape = circle, draw,fill=black] {}; 
\node[vertex] (G--2) at (1.5, -1) [shape = circle, draw,fill=black] {}; 
\node[vertex] (G-4) at (4.5, 1) [shape = circle, draw,fill=black] {}; 
\node[vertex] (G--1) at (0.0, -1) [shape = circle, draw,fill=black] {}; 
\node[vertex] (G-1) at (0.0, 1) [shape = circle, draw,fill=black] {}; 
\node[vertex] (G-2) at (1.5, 1) [shape = circle, draw,fill=black] {}; 
\node[vertex] (G-3) at (3.0, 1) [shape = circle, draw,fill=black] {}; 
\node[vertex] (G-6) at (7.5, 1) [shape = circle, draw,fill=black] {}; 
\node[vertex] (G-7) at (9.0, 1) [shape = circle, draw,fill=black] {}; 
\draw (G-5) .. controls +(1, -1) and +(-1, 1) .. (G--7); 
\draw (G--6) .. controls +(-0.8, 0.8) and +(0.8, 0.8) .. (G--3); 
\draw (G--5) .. controls +(-0.5, 0.5) and +(0.5, 0.5) .. (G--4); 
\draw (G-4) .. controls +(-1, -1) and +(1, 1) .. (G--2); 
\draw (G-1) .. controls +(0, -1) and +(0, 1) .. (G--1); 
\draw (G-2) .. controls +(0.5, -0.5) and +(-0.5, -0.5) .. (G-3); 
\draw (G-6) .. controls +(0.5, -0.5) and +(-0.5, -0.5) .. (G-7); 
\end{tikzpicture}
} \; .
\]
It is clear that replacing $t \otimes 1_1$ in the middle of the above
stacked product by $t \otimes \omega_1$ makes no difference in the product.

The next task is to identify a set of generators for
$\PTL_k(\delta)$. To that end, we set
\begin{equation}\label{e:eps_i}
\e_i := \tilde{e}_i = (1-p_i)e_i(1-p_i) .
\end{equation}
We note that the equality $\e_i = (1-p_i)e_i(1-p_i)$ remains valid if
we replace either or both factors of $(1-p_i)$ by $(1-p_{i+1})$.
Since the diagrams $l_i$, $r_i$, and $p_i$ have no horizontal edges,
it is clear that $\tilde{l}_i = l_i$, $\tilde{r}_i = r_i$, and
$\tilde{p}_i = p_i$.

\begin{thm}\label{t:generators}
$\PTL_k(\delta)$ is generated by the set $\{l_i, r_i, \e_i \mid i \in
  [k-1] \}$.
\end{thm}

\begin{proof}
Let $P$ be the subalgebra generated by the given set.  By Theorem
\ref{t:bases}, it suffices to show that $\tilde{d}$ belongs to $P$,
for any $d$ in $\D(k)$. Let $d = d(A,t,B)$ as discussed in the
paragraph after Lemma~\ref{l:triples}. By
Lemma~\ref{l:Motzkin-factors}, we have $d = r_A(t \otimes \omega_{k-n})
\otimes l_B$.

We will need to distinguish generators with differing number of
vertices, so we temporarily (in this proof only) write $e_i(n)$ for
the diagram $e_i$ in $\D(n)$, for each $n \le k$.  For any $1 \le n
\le k-1$, let
\[
e_{i,n}(k) = e_i(n) \otimes \omega_{k-n} = e_i(k) p_{n+1} \cdots p_k .
\]
Then by Corollary~\ref{c:tilde-cor}, $\tilde{e}_{i,n}(k) =
\tilde{e}_i(k) p_{n+1} \cdots p_k = \e_i p_{n+1} \cdots p_k$, hence
belongs to $P$. Now let
\[
t = e_{i_1}(n) e_{i_2}(n) \cdots e_{i_m}(n)
\]
be any word that expresses $t$ in terms of the standard
Temperley--Lieb generators of $\TL_n(\delta)$. Then the multiplication
rule in Theorem~\ref{t:bar-mult}(b) along with
Corollary~\ref{c:tilde-cor} implies that the equation
\[
r_A \tilde{e}_{i_1,n}(k) \tilde{e}_{i_2,n}(k) \cdots
\tilde{e}_{i_m,n}(k) l_B = \tilde{d} + \text{lower order terms}
\]
holds in $\PTL_k(\delta)$. The left hand side of the above equation
belongs to~$P$. The lower order terms are a linear combination of
$\tilde{u}$ such that $u$ belongs to $\D_j(k)$ for some $j<n$. By
induction, we may assume that each such $\tilde{u}$ belongs to
$P$. Hence, the same conclusion holds for~$\tilde{d}$, completing the
proof.
\end{proof}

\section{Semisimplicity of $\PTL_k(\delta)$}\label{s:ss}%
\noindent
In this section, we fix the ground ring $\Bbbk$, $\delta \in \Bbbk$,
and $k$.

\begin{thm}\label{t:Xn-dec}
Let $\Bbbk$ be an arbitrary unital commutative ring. Let $X(n)$ be the
$\Bbbk$-linear span of $\{\bar{d} \mid d \in \D_n(k)\}$. Then the
algebra $\PTL_k(\delta)$ has the direct sum decomposition
\begin{equation*}
  \PTL_k(\delta) = \bigoplus_{n=0}^k X(n)
\end{equation*}
into pairwise orthogonal two-sided ideals, in the sense that
$xy=0$ if $x \in X(n)$ and $y \in X(n')$, where $n \ne n'$.
\end{thm}

\begin{proof}
This is an immediate consequence of Theorem \ref{t:bar-mult}, along
with the fact that we are dealing only with balanced diagrams.
\end{proof}

Let $Q_n$ be the free $\Bbbk$-module on the set of all cardinality $n$
subsets of $[k]$. Identify $\End_\Bbbk(Q_n)$ with the matrix algebra
$\Mat_{\binom{k}{n}}(\Bbbk)$ by means of its basis. All tensor
products in this paper are taken over the ground ring $\Bbbk$, so we
write $\otimes$ instead of $\otimes_\Bbbk$.

\begin{thm}\label{t:Xn}
For any $\delta \in \Bbbk$, where $\Bbbk$ is a commutative unital
ring, the ideal $X(n)$ is isomorphic to the matrix algebra
\[
\Mat_{\binom{k}{n}}\big(\TL_n(\delta-1)\big) \cong
\Mat_{\binom{k}{n}}(\Bbbk) \otimes \TL_n(\delta-1).
\]
Furthermore, every $X(n)$-module is isomorphic to one of the form $Q_n
\otimes N$, where $N$ is a $\TL_n(\delta-1)$-module.
\end{thm}

\begin{proof}
Let $(A_i,d_i',B_i)$ be the triples in Lemma~\ref{l:triples}
corresponding to diagrams $d_i \in \D_n(k)$, for $i = 1,2$.
Theorem~\ref{t:bar-mult} implies that
\[
\bar{d}_1 \bar{d}_2 =
\begin{cases}
  (\delta-1)^N \bar{d}_3 = & \text{ if
  } B_1 = A_2 \\ 0 & \text{ otherwise}
\end{cases}
\]
where, in the nonzero case, $d'_3 = d'_1 \circ d'_2$ in the partition
monoid $\Ptn_n$ and the triple corresponding to $d_3 \in \D_n(k)$ is
$(A_1,d'_3, B_2)$. Notice that $d'_1 d'_2 = (\delta-1)^N d'_3$
in the Temperley--Lieb algebra $\TL_n(\delta-1)$.  This proves
the first claim. The second claim follows from the well known
isomorphism
\[
\Mat_m\big(\mathcal{A}) \cong \Mat_m(\Bbbk) \otimes \mathcal{A}
\]
as $\Bbbk$-algebras, for any $\Bbbk$-algebra $\mathcal{A}$ and any
$m$. The isomorphism is given for any $a_{ij} \in \mathcal{A}$ by
$\sum a_{ij} e_{ij} \mapsto \sum e_{ij} \otimes a_{ij}$, where
$e_{ij}$ is the $(i,j)$th matrix unit. The final claim is a standard
feature of Morita theory. The action of
\[
\Mat_{\binom{k}{n}}(\Bbbk) \otimes \TL_n(\delta-1)
\cong \End_\Bbbk(Q_n) \otimes \TL_n(\delta-1)
\]
on $Q_n \otimes N$ is the obvious one, in which $(f\otimes t)(A
\otimes v) = f(A) \otimes tv$, for any $f \in \End_\Bbbk(Q_n)$, $t \in
\TL_n(\delta-1)$, $A \in Q_n$, $v \in N$.
\end{proof}

\begin{rmk}\label{r:Morita}
(i) Theorem~\ref{t:Xn} implies that $X(n)$ is Morita equivalent to
  $\TL_n(\delta)$; that is, there is a category equivalence between
  their (left, or, equivalently, right) modules. In particular, this
  implies (rather trivially) that $\PTL_k(\delta)$ is an iterated
  inflation of Temperley--Lieb algebras, in the sense of
  \cites{KX:99,KX:01,Green-Paget}.

(ii) The bijection in Lemma~\ref{l:triples} between $\D_n(k)$ and
  triples induces an isomorphism $X(n) \cong Q_n \otimes
  \TL_n(\delta-1) \otimes Q_n^*$, where $Q_n^* := \Hom_\Bbbk(Q_n,
  \Bbbk)$ is the linear dual of $Q_n$.  This is an isomorphism of
  $\Bbbk$-modules, and also an isomorphism of $\Bbbk$-algebras, where
  $Q_n^* \otimes Q_n$ is identified with $\End_\Bbbk(Q_n)$ by the
  usual isomorphism.
\end{rmk}

\begin{thm}\label{t:ss}
Suppose that $\Bbbk$ is a field. Then:
  \begin{enumerate}
  \item $X(n)$ is semisimple if and only if the same is true of
    $\TL_n(\delta-1)$.
  \item If $X(n)$ is semisimple, its simple modules are of the form
    $Q_n \otimes \TL^\lambda$, where $\TL^\lambda$ is a simple
    $\TL_n(\delta-1)$-module.
  \item $\PTL_k(\delta)$ is semisimple for any $\delta \in
    \Bbbk$ such that $\TL_n(\delta-1)$ is semisimple for all
    $n=0,1,\dots, k$.
  \end{enumerate}
\end{thm}

\begin{proof}
(a) follows from Theorem \ref{t:Xn}; see e.g.~\cite{Munn}*{Lemma 4.5}.
  Part (b) is clear from the isomorphism in Theorem~\ref{t:Xn}, and
  part (c) follows from part (a).  
\end{proof}

\begin{rmk}\label{r:ss}
By Theorem~\ref{t:ss-TL}, if $q$ is a nonzero element of the field
$\Bbbk$ such that $\Brackets{n}^!_q \ne 0$, then $\PTL_n(1
\pm(q+q^{-1}))$ is semisimple over $\Bbbk$. In particular, this holds
whenever $q$ is not a root of unity.
\end{rmk}

\section{Representations of $\TL_k(\delta)$}\label{s:TL-mod}\noindent
The purpose of this section is to recall \cite{GL:96}*{Example 1.4}
the standard combinatorial construction of the cell modules for a
Temperley--Lieb algebra $\TL_k(\delta)$. A \emph{planar involution}
with $\lambda$ fixed points is a sequence of $k$ points arranged in a
line with $k-\lambda$ of them joined in pairs and $\lambda$ of them
with an ``end'' attached. The pairs correspond to interchanges and
ends to fixed points. The planar condition requires that the
involution can be drawn in a half-plane if the ends are extended to
infinity, without intersections. For example, the two diagrams
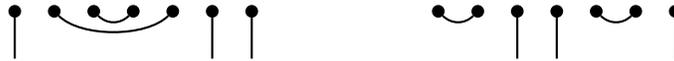
\[
\begin{tikzpicture}[scale = 0.35,thick, baseline={(0,-1.1ex/2)}] 
\tikzstyle{vertex} = [shape = circle, minimum size = 4pt, inner sep = 1pt] 
\node[vertex] (G-1) at (0.0, 0) [shape = circle, draw,fill=black] {}; 
\node[vertex] (G-2) at (1.5, 0) [shape = circle, draw, fill=black] {}; 
\node[vertex] (G-3) at (3.0, 0) [shape = circle, draw, fill=black] {};
\node[vertex] (G-4) at (4.5, 0) [shape = circle, draw, fill=black] {};
\node[vertex] (G-5) at (6.0, 0) [shape = circle, draw, fill=black] {};
\node[vertex] (G-6) at (7.5, 0) [shape = circle, draw, fill=black] {};
\node[vertex] (G-7) at (9.0, 0) [shape = circle, draw,fill=black] {};
\draw (G-3) .. controls +(0.5, -0.5) and +(-0.5, -0.5) .. (G-4);
\draw (G-2) .. controls +(1, -1) and +(-1, -1) .. (G-5);
\draw (G-1) -- (0.0,-1.8);
\draw (G-6) -- (7.5,-1.8);
\draw (G-7) -- (9.0,-1.8);
\end{tikzpicture}
\qquad\qquad\qquad
\begin{tikzpicture}[scale = 0.35,thick, baseline={(0,-1.1ex/2)}] 
\tikzstyle{vertex} = [shape = circle, minimum size = 4pt, inner sep = 1pt] 
\node[vertex] (G-1) at (0.0, 0) [shape = circle, draw,fill=black] {}; 
\node[vertex] (G-2) at (1.5, 0) [shape = circle, draw, fill=black] {}; 
\node[vertex] (G-3) at (3.0, 0) [shape = circle, draw, fill=black] {};
\node[vertex] (G-4) at (4.5, 0) [shape = circle, draw, fill=black] {};
\node[vertex] (G-5) at (6.0, 0) [shape = circle, draw, fill=black] {};
\node[vertex] (G-6) at (7.5, 0) [shape = circle, draw, fill=black] {};
\node[vertex] (G-7) at (9.0, 0) [shape = circle, draw,fill=black] {};
\draw (G-1) .. controls +(0.5, -0.5) and +(-0.5, -0.5) .. (G-2);
\draw (G-5) .. controls +(0.5, -0.5) and +(-0.5, -0.5) .. (G-6);
\draw (G-3) -- (3.0,-1.8);
\draw (G-4) -- (4.5,-1.8);
\draw (G-7) -- (9.0,-1.8);
\end{tikzpicture}
\]
depict planar involutions on $k=7$ points with $3$ fixed points. One
can think of planar involutions as products of disjoint cycles of
length $\le 2$, or as ``half'' Temperley--Lieb diagrams. Given two
planar involutions with the same number of fixed points, there is a
unique way to join the ends to create a Temperley--Lieb diagram. For
instance, the above involutions join to make the Temperley--Lieb
diagram
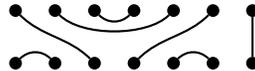
\[
\begin{tikzpicture}[scale = 0.35,thick, baseline={(0,-1ex/2)}] 
\tikzstyle{vertex} = [shape = circle, minimum size = 4pt, inner sep = 1pt] 
\node[vertex] (G--7) at (9.0, -1) [shape = circle, draw,fill=black] {}; 
\node[vertex] (G-7) at (9.0, 1) [shape = circle, draw,fill=black] {}; 
\node[vertex] (G--6) at (7.5, -1) [shape = circle, draw,fill=black] {}; 
\node[vertex] (G--5) at (6.0, -1) [shape = circle, draw,fill=black] {}; 
\node[vertex] (G--4) at (4.5, -1) [shape = circle, draw,fill=black] {}; 
\node[vertex] (G-6) at (7.5, 1) [shape = circle, draw,fill=black] {}; 
\node[vertex] (G--3) at (3.0, -1) [shape = circle, draw,fill=black] {}; 
\node[vertex] (G-1) at (0.0, 1) [shape = circle, draw,fill=black] {}; 
\node[vertex] (G--2) at (1.5, -1) [shape = circle, draw,fill=black] {}; 
\node[vertex] (G--1) at (0.0, -1) [shape = circle, draw,fill=black] {}; 
\node[vertex] (G-2) at (1.5, 1) [shape = circle, draw,fill=black] {}; 
\node[vertex] (G-5) at (6.0, 1) [shape = circle, draw,fill=black] {}; 
\node[vertex] (G-3) at (3.0, 1) [shape = circle, draw,fill=black] {}; 
\node[vertex] (G-4) at (4.5, 1) [shape = circle, draw,fill=black] {}; 
\draw[] (G-7) .. controls +(0, -1) and +(0, 1) .. (G--7); 
\draw[] (G--6) .. controls +(-0.5, 0.5) and +(0.5, 0.5) .. (G--5); 
\draw[] (G-6) .. controls +(-1, -1) and +(1, 1) .. (G--4); 
\draw[] (G-1) .. controls +(1, -1) and +(-1, 1) .. (G--3); 
\draw[] (G--2) .. controls +(-0.5, 0.5) and +(0.5, 0.5) .. (G--1); 
\draw[] (G-2) .. controls +(1, -1) and +(-1, -1) .. (G-5); 
\draw[] (G-3) .. controls +(0.5, -0.5) and +(-0.5, -0.5) .. (G-4); 
\end{tikzpicture}
\]
with the leftmost involution at the top and rightmost involution
(inverted) at the bottom.

By labeling the fixed points and left endpoint of interchanges by $1$
and labeling right endpoints of interchanges by $-1$, a planar
involution of length $k$ produces a sequence $a=(a_1,\dots,a_k)$
such that
\begin{enumerate}\renewcommand{\labelenumi}{(\roman{enumi})}
\item $a_i = \pm 1$ for all $1 \le i \le k$.
\item Each partial sum $a_1 + \cdots + a_n \ge 0$ for all $1 \le n \le
  k$.
\end{enumerate}
We call any sequence satisfying these conditions a
\emph{Temperley--Lieb path}.  Given any Temperley--Lieb path, we can
construct a unique planar involution that produces the sequence.  To
see this, pair each $i$ such that some $j>i$ exists with
$a_i+a_{i+1}+\cdots + a_j = 0$ with the unique minimal such~$j$. Such
pairings define the interchanges in the involution, and all unpaired
vertices are fixed points. Thus, we have a bijection between the sets
of planar involutions and Temperley--Lieb paths of the same length, so
we may as well identify these sets.

A Temperley--Lieb diagram $t$ acts on a planar involution $a$ by the
usual diagram multiplication, producing a multiple of another planar
involution $b$ having at most as many fixed points as $a$. So the free
$\Bbbk$-module $W_0$ on the set of planar involutions of length $k$ is
a $\TL_k(\delta)$-module.  Let $W_0^{\le \lambda}$ (resp.,
$W_0^{<\lambda}$) denote the span of the planar involutions with at
most (resp., fewer than) $\lambda$ fixed points. Then
\[
\TL^\lambda:= W_0^{\le \lambda}/W_0^{<\lambda}
\]
is a $\TL_k(\delta)$-module, for any $\lambda$ such that $k-\lambda$
is an even number.

\begin{thm}[\cite{GL:96}] 
Let $\Bbbk$ be a commutative unital ring. The collection
$\{\TL^\lambda \mid \lambda \equiv k~(\text{\upshape{mod} } 2)\}$ is a
complete set of cell modules for $\TL_k(\delta)$. If $\Bbbk$ is a
field and $\delta \in \Bbbk$ is chosen such that $\TL_k(\delta)$ is
semisimple, then the same set is a complete set of
isomorphism classes of simple $\TL_k(\delta)$-modules.
\end{thm}

\begin{rmk}\label{r:TL-indexing}
There is another way to index the cell modules for $\TL_k(\delta)$, by
the set $\Lambda_k$ of partitions of $k$ with at most two parts.  The
map $(\lambda_1,\lambda_2) \mapsto \lambda_1-\lambda_2$ defines a
bijection of $\Lambda_k$ with the set of integers in $\{0,1, \dots,
k\}$ which are congruent to $k$ mod $2$, where the value $\lambda_2=0$
is allowed. In this indexing scheme, for
$\lambda=(\lambda_1,\lambda_2)$ in $\Lambda_k$, the notation
$\TL^\lambda = W_0^{\lessdom \lambda}/W_0^{\slessdom \lambda}$
replaces the previous notation $\TL^m = W_0^{\le m}/W_0^{<m}$, where
$m = \lambda_1-\lambda_2$. Here $\dom$ is the usual dominance order on
partitions. We will switch to the partition notation starting in
\S\ref{s:reps}.
\end{rmk}

\section{Motzkin paths and representations}\label{s:Motz}\noindent
We now recall the combinatorial construction of the cell modules for
the Motzkin algebra, which generalizes the previous section. This will
be applied in the next section to construct the cell modules for
$\PTL_k(\delta)$. The remainder of this section closely follows
\cite{BH}*{\S4}, to which the reader should refer for additional
details.

A \emph{Motzkin path} of length $k$ is a sequence $a=(a_1,
\dots, a_k)$ such that $a_i \in \{-1,0,1\}$ and each partial sum $a_1
+ \cdots + a_n \ge 0$ for all $1 \le n \le k$. The \emph{rank} of a
Motzkin path $a=(a_i)$ is defined to be
\[
\rank(a):= a_1+\cdots+a_k.
\]
For each index $i$ with $a_i=1$, let $j$ be the smallest index (if
any) such that $i<j \le k$ and $a_i+a_{i+1}+\cdots+a_j=0$. Whenever
this happens, $(a_i,a_j) = (1,-1)$ are said to be \emph{paired};
otherwise, $a_i=1$ is \emph{unpaired}. By omitting the zeros in a
Motzkin path we obtain a Temperley--Lieb path in the sense of the
previous section.  Connecting paired indices by an edge and extending
unpaired indices by a line to infinity, we recover the planar
involution of that path. By including the discarded zeros as isolated
vertices, we obtain a graph on $[k]$ called a $1$-\emph{factor}.
Since paired indices cancel one another in the sum,
\[
\rank(a) = \text{ the number of fixed points (lines to infinity)}
\]
in the corresponding $1$-factor. 
For example, the graph
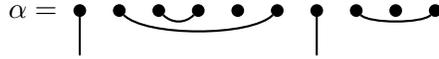
\[\alpha = \;
\begin{tikzpicture}[scale = 0.35,thick, baseline={(0,-1.1ex/2)}] 
\tikzstyle{vertex} = [shape = circle, minimum size = 4pt, inner sep = 1pt] 
\node[vertex] (G-1) at (0.0, 0) [shape = circle, draw,fill=black] {}; 
\node[vertex] (G-2) at (1.5, 0) [shape = circle, draw, fill=black] {}; 
\node[vertex] (G-3) at (3.0, 0) [shape = circle, draw, fill=black] {};
\node[vertex] (G-4) at (4.5, 0) [shape = circle, draw, fill=black] {};
\node[vertex] (G-5) at (6.0, 0) [shape = circle, draw, fill=black] {};
\node[vertex] (G-6) at (7.5, 0) [shape = circle, draw, fill=black] {};
\node[vertex] (G-7) at (9.0, 0) [shape = circle, draw,fill=black] {};
\node[vertex] (G-8) at (10.5, 0) [shape = circle, draw, fill=black] {};
\node[vertex] (G-9) at (12.0, 0) [shape = circle, draw, fill=black] {};
\node[vertex] (G-10) at (13.5, 0) [shape = circle, draw, fill=black] {};
\draw (G-3) .. controls +(0.5, -0.5) and +(-0.5, -0.5) .. (G-4);
\draw (G-8) .. controls +(0.5, -0.5) and +(-0.5, -0.5) .. (G-10);
\draw (G-2) .. controls +(1, -1) and +(-1, -1) .. (G-6);
\draw (G-1) -- (0,-1.7);
\draw (G-7) -- (9.0,-1.7);
\end{tikzpicture}
\]
is the $1$-factor produced by the Motzkin path $(1,1,1,-1,0,-1,1,1,0,-1)$.

In general, a $1$-factor on $k$ vertices is a graph that gives a
planar involution once its isolated points are removed. This means that
fixed points of the involution may not appear between paired vertices.
Labeling paired vertices in a $1$-factor $\alpha$ by $1,-1$
respectively, labeling fixed points by $1$, and labeling all other
vertices by $0$ produces a Motzkin path $a$ given by the sequence of
labels whose corresponding $1$-factor is $\alpha$.  So there is a
bijection between Motzkin paths and $1$-factors of the same length.
From now on, we identify Motzkin paths with $1$-factors
by means of this bijection.

\begin{rmk}
In \cite{BH}, the fixed points of a $1$-factor are depicted by
white-colored vertices instead of lines to infinity.
\end{rmk}

For any given pair $(\alpha,\beta)$, where $\alpha,\beta$ are
$1$-factors on $k$ vertices of the same rank $\lambda$, there is a
unique Motzkin $k$-diagram $C^\lambda_{\alpha,\,\beta}$ such that the
fixed points in $\alpha$ are connected to those in $\beta$. For
example, if $\alpha$ is the $1$-factor displayed above and $\beta$ the
$1$-factor (of rank $2$) below
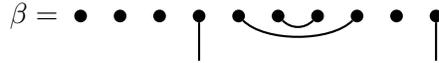
\begin{align*}
\beta &= \;
\begin{tikzpicture}[scale = 0.35,thick, baseline={(0,-1.1ex/2)}] 
\tikzstyle{vertex} = [shape = circle, minimum size = 4pt, inner sep = 1pt] 
\node[vertex] (G-1) at (0.0, 0) [shape = circle, draw, fill=black] {}; 
\node[vertex] (G-2) at (1.5, 0) [shape = circle, draw, fill=black] {}; 
\node[vertex] (G-3) at (3.0, 0) [shape = circle, draw, fill=black] {};
\node[vertex] (G-4) at (4.5, 0) [shape = circle, draw,fill=black] {};
\node[vertex] (G-5) at (6.0, 0) [shape = circle, draw, fill=black] {};
\node[vertex] (G-6) at (7.5, 0) [shape = circle, draw, fill=black] {};
\node[vertex] (G-7) at (9.0, 0) [shape = circle, draw, fill=black] {};
\node[vertex] (G-8) at (10.5, 0) [shape = circle, draw, fill=black] {};
\node[vertex] (G-9) at (12.0, 0) [shape = circle, draw,fill=black] {};
\node[vertex] (G-10) at (13.5, 0) [shape = circle, draw,fill=black] {};
\draw (G-6) .. controls +(0.5, -0.5) and +(-0.5, -0.5) .. (G-7);
\draw (G-5) .. controls +(1, -1) and +(-1, -1) .. (G-8);
\draw (G-4) -- (4.5,-1.7);
\draw (G-10) -- (13.5,-1.7);
\end{tikzpicture} \\
\intertext{%
then $C^\lambda_{\alpha,\beta}$ is obtained by reflecting $\beta$ across its
horizontal axis and then drawing edges connecting the fixed points
in $\alpha$, $\beta$ in order, which gives}
C_{\alpha,\beta}^2 &= \; \shade
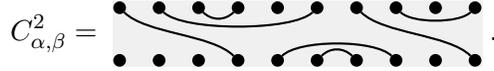
{
\begin{tikzpicture}[scale = 0.35,thick, baseline={(0,-1ex/2)}] 
\tikzstyle{vertex} = [shape = circle, minimum size = 4pt, inner sep = 1pt] 
\node[vertex] (G--10) at (13.5, -1) [shape = circle, draw,fill=black] {}; 
\node[vertex] (G-7) at (9.0, 1) [shape = circle, draw,fill=black] {}; 
\node[vertex] (G--9) at (12.0, -1) [shape = circle, draw,fill=black] {}; 
\node[vertex] (G--8) at (10.5, -1) [shape = circle, draw,fill=black] {}; 
\node[vertex] (G--5) at (6.0, -1) [shape = circle, draw,fill=black] {}; 
\node[vertex] (G--7) at (9.0, -1) [shape = circle, draw,fill=black] {}; 
\node[vertex] (G--6) at (7.5, -1) [shape = circle, draw,fill=black] {}; 
\node[vertex] (G--4) at (4.5, -1) [shape = circle, draw,fill=black] {}; 
\node[vertex] (G-1) at (0.0, 1) [shape = circle, draw,fill=black] {}; 
\node[vertex] (G--3) at (3.0, -1) [shape = circle, draw,fill=black] {}; 
\node[vertex] (G--2) at (1.5, -1) [shape = circle, draw,fill=black] {}; 
\node[vertex] (G--1) at (0.0, -1) [shape = circle, draw,fill=black] {}; 
\node[vertex] (G-2) at (1.5, 1) [shape = circle, draw,fill=black] {}; 
\node[vertex] (G-6) at (7.5, 1) [shape = circle, draw,fill=black] {}; 
\node[vertex] (G-3) at (3.0, 1) [shape = circle, draw,fill=black] {}; 
\node[vertex] (G-4) at (4.5, 1) [shape = circle, draw,fill=black] {}; 
\node[vertex] (G-5) at (6.0, 1) [shape = circle, draw,fill=black] {}; 
\node[vertex] (G-8) at (10.5, 1) [shape = circle, draw,fill=black] {}; 
\node[vertex] (G-10) at (13.5, 1) [shape = circle, draw,fill=black] {}; 
\node[vertex] (G-9) at (12.0, 1) [shape = circle, draw,fill=black] {}; 
\draw (G-7) .. controls +(1, -1) and +(-1, 1) .. (G--10); 
\draw (G--8) .. controls +(-0.8, 0.8) and +(0.8, 0.8) .. (G--5); 
\draw (G--7) .. controls +(-0.5, 0.5) and +(0.5, 0.5) .. (G--6); 
\draw (G-1) .. controls +(1, -1) and +(-1, 1) .. (G--4); 
\draw (G-2) .. controls +(0.9, -0.9) and +(-0.9, -0.9) .. (G-6); 
\draw (G-3) .. controls +(0.5, -0.5) and +(-0.5, -0.5) .. (G-4); 
\draw (G-8) .. controls +(0.6, -0.6) and +(-0.6, -0.6) .. (G-10); 
\end{tikzpicture}
} \; .
\end{align*}
Notice that the zeros in the Motzkin paths corresponding to
$\alpha,\beta$ label the isolated vertices in the diagram
$C^\lambda_{\alpha,\beta}$.  For any $k$, the disjoint
union
\begin{equation}
  \textstyle \bigsqcup_{\lambda=0}^k \{ C^\lambda_{\alpha,\,\beta}
  \mid \rank(\alpha) = \rank(\beta)=\lambda\}
\end{equation}
is a cellular basis of $\M_k(\delta)$. This is the basis of Motzkin
$k$-diagrams. 

Now we describe the cell modules for the Motzkin algebra.  For a
Motzkin diagram $d$ and a Motzkin path $a$ (viewed as a $1$-factor) of
the same length, the Motzkin path $da$ is given by
\begin{equation}\label{e:path-action}
  da = \delta^{N(d,a)} b
\end{equation}
where the multiplication is carried out by the usual graphical
stacking procedure. The integer $N(d,a)$ is the number of loops in the
bottom row of $\Gamma(d,a)$, and $b$ is the resulting $1$-factor in
the top row of $\Gamma(d,a)$, after erasing the second row of
vertices.

For example, if $d$ is the Motzkin diagram and $a$ the
$1$-factor given by the pictures
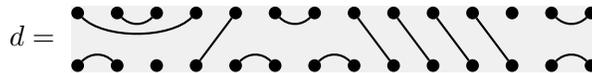
\begin{align*}
d &= \;
\shade{\begin{tikzpicture}[scale = 0.35,thick, baseline={(0,-1ex/2)}] 
\tikzstyle{vertex} = [shape = circle, minimum size = 4pt, inner sep = 1pt] 
\node[vertex] (G--14) at (19.5, -1) [shape = circle, draw,fill=black] {}; 
\node[vertex] (G--13) at (18.0, -1) [shape = circle, draw,fill=black] {}; 
\node[vertex] (G--12) at (16.5, -1) [shape = circle, draw,fill=black] {}; 
\node[vertex] (G-11) at (15.0, 1) [shape = circle, draw,fill=black] {}; 
\node[vertex] (G--11) at (15.0, -1) [shape = circle, draw,fill=black] {}; 
\node[vertex] (G-10) at (13.5, 1) [shape = circle, draw,fill=black] {}; 
\node[vertex] (G--10) at (13.5, -1) [shape = circle, draw,fill=black] {}; 
\node[vertex] (G-9) at (12.0, 1) [shape = circle, draw,fill=black] {}; 
\node[vertex] (G--9) at (12.0, -1) [shape = circle, draw,fill=black] {}; 
\node[vertex] (G-8) at (10.5, 1) [shape = circle, draw,fill=black] {}; 
\node[vertex] (G--8) at (10.5, -1) [shape = circle, draw,fill=black] {}; 
\node[vertex] (G--7) at (9.0, -1) [shape = circle, draw,fill=black] {}; 
\node[vertex] (G--6) at (7.5, -1) [shape = circle, draw,fill=black] {}; 
\node[vertex] (G--5) at (6.0, -1) [shape = circle, draw,fill=black] {}; 
\node[vertex] (G--4) at (4.5, -1) [shape = circle, draw,fill=black] {}; 
\node[vertex] (G-5) at (6.0, 1) [shape = circle, draw,fill=black] {}; 
\node[vertex] (G--3) at (3.0, -1) [shape = circle, draw,fill=black] {}; 
\node[vertex] (G--2) at (1.5, -1) [shape = circle, draw,fill=black] {}; 
\node[vertex] (G--1) at (0.0, -1) [shape = circle, draw,fill=black] {}; 
\node[vertex] (G-1) at (0.0, 1) [shape = circle, draw,fill=black] {}; 
\node[vertex] (G-4) at (4.5, 1) [shape = circle, draw,fill=black] {}; 
\node[vertex] (G-2) at (1.5, 1) [shape = circle, draw,fill=black] {}; 
\node[vertex] (G-3) at (3.0, 1) [shape = circle, draw,fill=black] {}; 
\node[vertex] (G-6) at (7.5, 1) [shape = circle, draw,fill=black] {}; 
\node[vertex] (G-7) at (9.0, 1) [shape = circle, draw,fill=black] {}; 
\node[vertex] (G-12) at (16.5, 1) [shape = circle, draw,fill=black] {}; 
\node[vertex] (G-13) at (18.0, 1) [shape = circle, draw,fill=black] {}; 
\node[vertex] (G-14) at (19.5, 1) [shape = circle, draw,fill=black] {}; 
\draw (G--14) .. controls +(-0.5, 0.5) and +(0.5, 0.5) .. (G--13); 
\draw (G-11) .. controls +(0.75, -1) and +(-0.75, 1) .. (G--12); 
\draw (G-10) .. controls +(0.75, -1) and +(-0.75, 1) .. (G--11); 
\draw (G-9) .. controls +(0.75, -1) and +(-0.75, 1) .. (G--10); 
\draw (G-8) .. controls +(0.75, -1) and +(-0.75, 1) .. (G--9); 
\draw (G--8) .. controls +(-0.5, 0.5) and +(0.5, 0.5) .. (G--7); 
\draw (G--6) .. controls +(-0.5, 0.5) and +(0.5, 0.5) .. (G--5); 
\draw (G-5) .. controls +(-0.75, -1) and +(0.75, 1) .. (G--4); 
\draw (G--2) .. controls +(-0.5, 0.5) and +(0.5, 0.5) .. (G--1); 
\draw (G-1) .. controls +(1, -1) and +(-1, -1) .. (G-4); 
\draw (G-2) .. controls +(0.5, -0.5) and +(-0.5, -0.5) .. (G-3); 
\draw (G-6) .. controls +(0.5, -0.5) and +(-0.5, -0.5) .. (G-7); 
\draw (G-13) .. controls +(0.5, -0.5) and +(-0.5, -0.5) .. (G-14); 
\end{tikzpicture} } \\
a &= \;
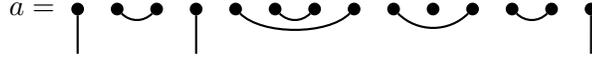
\begin{tikzpicture}[scale = 0.35,thick, baseline={(0,-1ex/2)}] 
\tikzstyle{vertex} = [shape = circle, minimum size = 4pt, inner sep = 1pt] 
\node[vertex] (G-11) at (15.0, 0) [shape = circle, draw,fill=black] {}; 
\node[vertex] (G-10) at (13.5, 0) [shape = circle, draw,fill=black] {}; 
\node[vertex] (G-9) at (12.0, 0) [shape = circle, draw,fill=black] {}; 
\node[vertex] (G-8) at (10.5, 0) [shape = circle, draw,fill=black] {}; 
\node[vertex] (G-5) at (6.0, 0) [shape = circle, draw,fill=black] {}; 
\node[vertex] (G-1) at (0.0, 0) [shape = circle, draw,fill=black] {}; 
\node[vertex] (G-4) at (4.5, 0) [shape = circle, draw,fill=black] {}; 
\node[vertex] (G-2) at (1.5, 0) [shape = circle, draw,fill=black] {}; 
\node[vertex] (G-3) at (3.0, 0) [shape = circle, draw,fill=black] {}; 
\node[vertex] (G-6) at (7.5, 0) [shape = circle, draw,fill=black] {}; 
\node[vertex] (G-7) at (9.0, 0) [shape = circle, draw,fill=black] {}; 
\node[vertex] (G-12) at (16.5, 0) [shape = circle, draw,fill=black] {}; 
\node[vertex] (G-13) at (18.0, 0) [shape = circle, draw,fill=black] {}; 
\node[vertex] (G-14) at (19.5, 0) [shape = circle, draw,fill=black] {}; 
\draw (G-5) .. controls +(1, -1) and +(-1, -1) .. (G-8); 
\draw (G-6) .. controls +(0.5, -0.5) and +(-0.5, -0.5) .. (G-7); 
\draw (G-2) .. controls +(0.5, -0.5) and +(-0.5, -0.5) .. (G-3); 
\draw (G-12) .. controls +(0.5, -0.5) and +(-0.5, -0.5) .. (G-13);
\draw (G-9) .. controls +(0.9, -0.9) and +(-0.9, -0.9) .. (G-11);
\draw (G-1) -- (0.0,-1.7);
\draw (G-4) -- (4.5,-1.7);
\draw (G-14) -- (19.5,-1.7);
\end{tikzpicture} 
\intertext{then the resulting $\Gamma(d,a)$ is the configuration}
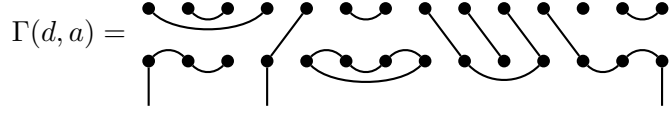
\Gamma(d,a) &= \;
\begin{tikzpicture}[scale = 0.35,thick, baseline={(0,-1ex/2)}] 
\tikzstyle{vertex} = [shape = circle, minimum size = 4pt, inner sep = 1pt] 
\node[vertex] (G--14) at (19.5, -1) [shape = circle, draw,fill=black] {}; 
\node[vertex] (G--13) at (18.0, -1) [shape = circle, draw,fill=black] {}; 
\node[vertex] (G--12) at (16.5, -1) [shape = circle, draw,fill=black] {}; 
\node[vertex] (G-11) at (15.0, 1) [shape = circle, draw,fill=black] {}; 
\node[vertex] (G--11) at (15.0, -1) [shape = circle, draw,fill=black] {}; 
\node[vertex] (G-10) at (13.5, 1) [shape = circle, draw,fill=black] {}; 
\node[vertex] (G--10) at (13.5, -1) [shape = circle, draw,fill=black] {}; 
\node[vertex] (G-9) at (12.0, 1) [shape = circle, draw,fill=black] {}; 
\node[vertex] (G--9) at (12.0, -1) [shape = circle, draw,fill=black] {}; 
\node[vertex] (G-8) at (10.5, 1) [shape = circle, draw,fill=black] {}; 
\node[vertex] (G--8) at (10.5, -1) [shape = circle, draw,fill=black] {}; 
\node[vertex] (G--7) at (9.0, -1) [shape = circle, draw,fill=black] {}; 
\node[vertex] (G--6) at (7.5, -1) [shape = circle, draw,fill=black] {}; 
\node[vertex] (G--5) at (6.0, -1) [shape = circle, draw,fill=black] {}; 
\node[vertex] (G--4) at (4.5, -1) [shape = circle, draw,fill=black] {}; 
\node[vertex] (G-5) at (6.0, 1) [shape = circle, draw,fill=black] {}; 
\node[vertex] (G--3) at (3.0, -1) [shape = circle, draw,fill=black] {}; 
\node[vertex] (G--2) at (1.5, -1) [shape = circle, draw,fill=black] {}; 
\node[vertex] (G--1) at (0.0, -1) [shape = circle, draw,fill=black] {}; 
\node[vertex] (G-1) at (0.0, 1) [shape = circle, draw,fill=black] {}; 
\node[vertex] (G-4) at (4.5, 1) [shape = circle, draw,fill=black] {}; 
\node[vertex] (G-2) at (1.5, 1) [shape = circle, draw,fill=black] {}; 
\node[vertex] (G-3) at (3.0, 1) [shape = circle, draw,fill=black] {}; 
\node[vertex] (G-6) at (7.5, 1) [shape = circle, draw,fill=black] {}; 
\node[vertex] (G-7) at (9.0, 1) [shape = circle, draw,fill=black] {}; 
\node[vertex] (G-12) at (16.5, 1) [shape = circle, draw,fill=black] {}; 
\node[vertex] (G-13) at (18.0, 1) [shape = circle, draw,fill=black] {}; 
\node[vertex] (G-14) at (19.5, 1) [shape = circle, draw,fill=black] {}; 
\draw (G--14) .. controls +(-0.5, 0.5) and +(0.5, 0.5) .. (G--13); 
\draw (G-11) .. controls +(0.75, -1) and +(-0.75, 1) .. (G--12); 
\draw (G-10) .. controls +(0.75, -1) and +(-0.75, 1) .. (G--11); 
\draw (G-9) .. controls +(0.75, -1) and +(-0.75, 1) .. (G--10); 
\draw (G-8) .. controls +(0.75, -1) and +(-0.75, 1) .. (G--9); 
\draw (G--8) .. controls +(-0.5, 0.5) and +(0.5, 0.5) .. (G--7); 
\draw (G--6) .. controls +(-0.5, 0.5) and +(0.5, 0.5) .. (G--5); 
\draw (G-5) .. controls +(-0.75, -1) and +(0.75, 1) .. (G--4); 
\draw (G--2) .. controls +(-0.5, 0.5) and +(0.5, 0.5) .. (G--1); 
\draw (G-1) .. controls +(1, -1) and +(-1, -1) .. (G-4); 
\draw (G-2) .. controls +(0.5, -0.5) and +(-0.5, -0.5) .. (G-3); 
\draw (G-6) .. controls +(0.5, -0.5) and +(-0.5, -0.5) .. (G-7); 
\draw (G-13) .. controls +(0.5, -0.5) and +(-0.5, -0.5) .. (G-14);
\draw (G--5) .. controls +(1, -1) and +(-1, -1) .. (G--8); 
\draw (G--6) .. controls +(0.5, -0.5) and +(-0.5, -0.5) .. (G--7); 
\draw (G--2) .. controls +(0.5, -0.5) and +(-0.5, -0.5) .. (G--3); 
\draw (G--12) .. controls +(0.5, -0.5) and +(-0.5, -0.5) .. (G--13);
\draw (G--9) .. controls +(0.9, -0.9) and +(-0.9, -0.9) .. (G--11);
\draw (G--1) -- (0.0,-2.7);
\draw (G--4) -- (4.5,-2.7);
\draw (G--14) -- (19.5,-2.7);
\end{tikzpicture}
\intertext{and thus $da = \delta\, b$, where $b$ is the $1$-factor}
b &= \;
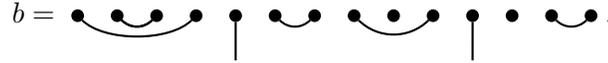
\begin{tikzpicture}[scale = 0.35,thick, baseline={(0,-1ex/2)}] 
\tikzstyle{vertex} = [shape = circle, minimum size = 4pt, inner sep = 1pt] 
\node[vertex] (G-11) at (15.0, 0) [shape = circle, draw,fill=black] {}; 
\node[vertex] (G-10) at (13.5, 0) [shape = circle, draw,fill=black] {}; 
\node[vertex] (G-9) at (12.0, 0) [shape = circle, draw,fill=black] {}; 
\node[vertex] (G-8) at (10.5, 0) [shape = circle, draw,fill=black] {}; 
\node[vertex] (G-5) at (6.0, 0) [shape = circle, draw,fill=black] {}; 
\node[vertex] (G-1) at (0.0, 0) [shape = circle, draw,fill=black] {}; 
\node[vertex] (G-4) at (4.5, 0) [shape = circle, draw,fill=black] {}; 
\node[vertex] (G-2) at (1.5, 0) [shape = circle, draw,fill=black] {}; 
\node[vertex] (G-3) at (3.0, 0) [shape = circle, draw,fill=black] {}; 
\node[vertex] (G-6) at (7.5, 0) [shape = circle, draw,fill=black] {}; 
\node[vertex] (G-7) at (9.0, 0) [shape = circle, draw,fill=black] {}; 
\node[vertex] (G-12) at (16.5, 0) [shape = circle, draw,fill=black] {}; 
\node[vertex] (G-13) at (18.0, 0) [shape = circle, draw,fill=black] {}; 
\node[vertex] (G-14) at (19.5, 0) [shape = circle, draw,fill=black] {}; 
\draw (G-1) .. controls +(1, -1) and +(-1, -1) .. (G-4); 
\draw (G-2) .. controls +(0.5, -0.5) and +(-0.5, -0.5) .. (G-3); 
\draw (G-2) .. controls +(0.5, -0.5) and +(-0.5, -0.5) .. (G-3); 
\draw (G-6) .. controls +(0.5, -0.5) and +(-0.5, -0.5) .. (G-7);
\draw (G-8) .. controls +(0.9, -0.9) and +(-0.9, -0.9) .. (G-10);
\draw (G-13) .. controls +(0.5, -0.5) and +(-0.5, -0.5) .. (G-14);
\draw (G-5) -- (6.0,-1.7);
\draw (G-11) -- (15.0,-1.7);
\end{tikzpicture}\; .
\end{align*}
There is a factor of $\delta$ in this example because there is a
single loop in $\Gamma(d,a)$.

Let $W$ be the free $\Bbbk$-module on the Motzkin paths of
length $k$.  The above action extends linearly to make $W$ into an
$\M_k(\delta)$-module. Since
\begin{equation}
  \rank(da) \le \min(\rank(d),\rank(a)), 
\end{equation}
the $\Bbbk$-submodule $W^{\le \lambda}$ of $W$ spanned by the Motzkin
paths of rank at most $\lambda$ is an $\M_k(\delta)$-submodule, for
any $\lambda = 0, 1, \dots, k$. Thus we have a filtration $(0)
\subseteq W^{\le 0} \subseteq W^{\le 1} \subseteq \cdots \subseteq
W^{\le k} = W$ of $\M_k(\delta)$-submodules. For each $\lambda$, the
quotient module
\[
\M^\lambda := W^{\le \lambda} / W^{<\lambda}
\]
is an $\M_k(\delta)$-module, where we set $W^{<\lambda} := W^{\le
  \lambda-1}$ (and $W^{<0} = (0)$).

\begin{thm}[\cite{BH}*{Thms.~4.7, 4.16}]\label{t:M-cell-reps}
Let $\Bbbk$ be a commutative unital ring. The cell modules for
$\M_k(\delta)$ are given by
\[
\{ \M^\lambda \mid 0 \le \lambda \le k\}.
\]
If $\Bbbk$ is a field then $\M^\lambda$ is indecomposable, and if
$\delta \in \Bbbk$ is chosen so that $\M_k(\delta)$ is semisimple then
the above set is a complete set of pairwise nonisomorphic simple
$\M_k(\delta)$-modules.
\end{thm}


By restricting our attention to the Motzkin paths with no zeros, we
obtain copies of the cell modules for $\TL_k(\delta)$. The following
was observed in \cite{BH}*{Rmk.~4.9}.

\begin{thm}\label{t:TL-cell-reps}
Let $\mathcal{T}^\lambda$ be the $\Bbbk$-span of the Motzkin paths
$a$ in $\M^\lambda$ with no zeros.  The span of the Motzkin diagrams
in $\M_k(\delta)$ having no isolated vertices is a subalgebra
isomorphic to $\TL_k(\delta)$.  When restricted to that subalgebra,
the action of $\M_k(\delta)$ on $\M^\lambda$ makes
$\mathcal{T}^\lambda$ into a $\TL_k(\delta)$-module which is
isomorphic to the cell module $\TL^\lambda$, for any $\lambda$ such that
$k-\lambda$ is even.
\end{thm}

\section{Representations of $\PTL_k(\delta)$}\label{s:reps}\noindent
We continue to fix $k$ and $\delta \in \Bbbk$, where $\Bbbk$ is a
given unital commutative ring.  Recall from Theorem \ref{t:Xn-dec}
that $\PTL_k(\delta) = \bigoplus_{n=0}^k X(n)$. From now on, we will
index representations by partitions of not more than two parts,
instead of by integers as in the previous two sections; see
Remark~\ref{r:TL-indexing}.

The \emph{type} of a Motzkin path $a = (a_1,\dots, a_k)$ is the pair
$\lambda = (\lambda_1,\lambda_2)$ such that $\lambda_1$ (resp.,
$\lambda_2$) be the number of $i$ such that $a_i = 1$ (resp., $a_i =
-1$). Then $\lambda_1 \ge \lambda_2 \ge 0$ and $\rank(a)=\lambda_1 -
\lambda_2$.  We set
\[
|\lambda| := \lambda_1+\lambda_2.
\]
The $1$-factor corresponding to $a$ has $\lambda_2$ cups and $a$ has
$k-|\lambda|$ zeros. Write
\[
\Lambda :=\{(\lambda_1,\lambda_2) \mid \lambda_i \in \Z, \lambda_1\ge
\lambda_2 \ge 0, \lambda_1+\lambda_2 \le k\}
\]
for the set of types that occur as the type of some Motzkin path.  We
identify elements of $\Lambda$ with partitions of at most two parts.
Thus $\Lambda$ is the disjoint union of the $\Lambda_n$ (see Remark
\ref{r:TL-indexing}) as $n$ runs from $0$ to $k$.

For Motzkin paths $a=(a_1,\dots,a_k)$, $a'=(a'_1,\dots, a'_k)$ we
write $a' \le a$ if the following two conditions are satisfied:
\begin{enumerate}\renewcommand{\labelenumi}{(\roman{enumi})}
\item $a_i = 0$ implies that $a'_i = 0$.
\item $a'_i = 0$ if and only if $a'_j=0$ whenever $a_i,a_j$ are paired in $a$.
\end{enumerate}
In other words, $a' \le a$ if and only if the $1$-factor corresponding
to $a'$ is obtainable from the $1$-factor corresponding to $a$ by
erasing zero or more edges or lines to infinity. Let $\chi(a,a')$ be
the number of such changes.  Define $\bar{a}$ to be the linear
combination
\begin{equation}
  \bar{a} :=  \sum_{a' \le a} (-1)^{\chi(a,a')} a' \, .
\end{equation}
For example, if $a=(1,-1,1)$ then
\[
\bar{a} = (1,-1,1) - (0,0,1) - (1,-1,0) + (0,0,0).
\]
If $i$ is any index such that $a_i = \pm 1$ then $p_i a$ is the unique
Motzkin path obtained from $a$ by changing $a_i$ to $0$ (and also
changing $a_j$ to $0$ if $a_i$, $a_j$ are paired).

We define $\topp{a}$ to be the the set of all $i$ such that $a_i = \pm
1$. In the corresponding $1$-factor, this set indexes the vertices
that are endpoints of edges or lines to infinity. Notice that we have
\begin{equation}\label{e:bar-a}
\bar{a} = \textstyle \prod_{i\in \topp{a}} (1-p_i) \, a
\end{equation}
in the Motzkin algebra. Notice that $\{\bar{a} \mid a \text{ is a
  Motzkin path of length $k$}\}$ is a $\Bbbk$-basis for the free
$\Bbbk$-module $W$ on the set of Motzkin paths of length $k$. The next
result explains how $\PTL_k(\delta)$ acts on $W$.

\begin{thm}\label{t:PTL}
Let $\Bbbk$ be a commutative unital ring.  Suppose that $a$ is a
Motzkin path in $W$ of type $\lambda$ in $\Lambda$. Let $b$ be the
unique Motzkin path such that $d a = \delta^N\, b$, for some $N$,
where $d \in \D(k)$ is balanced. 
\[
\bar{d} \, \bar{a} =
\begin{cases}
  (\delta-1)^{N(d,a)}\, \bar{b} & \text{ if $\bott{d} = \topp{a}$}\\
  \phantom{\delta-1} 0   & \text{ otherwise.}
\end{cases}
\]
If $\bott{d} = \topp{a}$ then the type of $b$ is some $\mu$ in
$\Lambda$ such that $|\mu|=|\lambda|$, and $\rank(a)-\rank(b)$ is an
even number.
\end{thm}

\begin{proof}
Thanks to the identity \eqref{e:bar-a}, the action of $\bar{d}$ by
$\bar{a}$ is given by precisely the same rule as the multiplication
rule in Theorem~\ref{t:bar-mult}(a), giving the first claim. Suppose
that $\bott{d} = \topp{a}$. Then the zeros in $a$ appear at the same
places at the isolated vertices in the bottom row of $d$. Since $d$ is
balanced, it has the same number of isolated vertices in its top row,
so $|\mu|=|\lambda|$. The rank of $b$ cannot be larger but may be
strictly less than that of~$a$. (For instance, if $k=3$ and
$a=(1,1,1)$ then $e_1 a = b = (1,-1,1)$, so $\bar{e}_1 \bar{a} =
\bar{b}$ where $\rank(a)=3$, $\rank(b)=1$.)  The rank decreases only
if one or more pairs of unpaired indices of $1$ in $a$ are replaced by
pairings in $b$, so the rank can decrease only in steps of~$2$.
\end{proof}

Recall that the usual \emph{dominance order} on partitions is defined
by declaring that $\lambda \dom \mu$ if and only if $\lambda-\mu$ can
be written as a sum of positive roots (in the root system of $\gl_n$).
In our situation, if $|\lambda|=|\mu|$ and $\lambda,\mu \in \Lambda$,
this is equivalent to
\[
\lambda \dom \mu \iff \lambda-\mu = m(1,-1) \text{ for some integer }
m \ge 0.
\]
This is equivalent to $(\lambda_1-\lambda_2) - (\mu_1-\mu_2)$ being a
nonnegative even integer. Write $\lambda \sdom \mu$ whenever $\lambda
\dom \mu$ but $\lambda \ne \mu$.

For any $\lambda$ in $\Lambda$, let $\ov{W}^{\lessdom \lambda}$ and
$\ov{W}^{\slessdom \lambda}$ be the $\Bbbk$-span of the sets
\[
\{ \bar{a} \mid a \text{ has type } \mu \text{ and } \lambda \dom
\mu\} \quad \text{and}\quad \{ \bar{a} \mid a \text{ has type } \mu
\text{ and } \lambda \sdom \mu\}
\]
respectively. By Theorem~\ref{t:PTL}, both $\ov{W}^{\lessdom \lambda}$,
$\ov{W}^{\slessdom \lambda}$ are $\PTL_k(\delta)$-submodules of $W$. We
define
\[
\PTL^\lambda := \ov{W}^{\lessdom \lambda}/\ov{W}^{\slessdom \lambda}
\]
to be the corresponding quotient module. So the collection of $\bar{a}
+ \ov{W}^{\slessdom \lambda}$ such that $a$ has type $\lambda$ is a
basis of $\PTL^\lambda$.

Theorem \ref{t:PTL} above and Theorems \ref{t:Xn}, \ref{t:ss}
imply the following result.

\begin{thm}\label{t:PTL-cell}
Let $\Bbbk$ be a commutative unital ring.  The $\PTL^\lambda$ for
$\lambda$ in $\Lambda$ are the cell modules for $\PTL_k(\delta)$.  For
any $\lambda \in \Lambda$, with $n = |\lambda|$ we have an isomorphism
\[
\PTL^\lambda \cong Q_n \otimes \TL^\lambda
\]
as $X(n)$-modules, where $\TL^\lambda$ is the cell module
for $\TL_n(\delta-1)$ indexed by $\lambda$ in accordance with
Remark~\textup{\ref{r:TL-indexing}}.  Hence, if $\Bbbk$ is a field and
$\delta \in \Bbbk$ is chosen so that $\PTL_k(\delta)$ is
semisimple, then
\[
\{ \PTL^\lambda \mid \lambda \in \Lambda \}
\]
is a complete set of pairwise nonisomorphic simple
$\PTL_k(\delta)$-modules.
\end{thm}

\begin{proof}
Suppose that $a$ is a Motzkin path of type $\lambda$, for $\lambda$ in
$\Lambda$. Let $A = \topp{a}$. If $n=|A|$ then $A$ is an element of
$Q_n$. Let $b$ be the Motzkin path obtained from $a$ by removing all
its zero entries; that is, $b$ is the part of $a$ supported by
$A$. Then $b$ may be regarded as a Temperley--Lieb half diagram (a
planar involution in the terminology of Graham and Lehrer; see
\S\ref{s:TL-mod}). Suppose that $|A|=n$. The linear map sending
\[
\bar{a} \mapsto A \otimes b
\]
defines an isomorphism $\ov{W}^{\lessdom \lambda} \to Q_n \otimes
W_0^{\lessdom \lambda}$ that restricts to an isomorphism
$\ov{W}^{\slessdom \lambda} \to Q_n \otimes W_0^{\slessdom
  \lambda}$. Passing to quotients induces the desired isomorphism
$\PTL^\lambda \cong Q_n \otimes \TL^\lambda$. This shows that the
$\PTL^\lambda$ are inflations of the various cell modules for
$\TL_n(\delta-1)$, which proves the first claim. The last claim
follows from the first two.
\end{proof}

As a consequence of the last result, if $\lambda$ is in $\Lambda$ and
satisfies $\lambda_1-\lambda_2 = n$, we have
\begin{equation}
  \rank_\Bbbk \PTL^\lambda = \binom{k}{n}\, \rank_\Bbbk \TL^\lambda.
\end{equation}

\begin{rmk}
The subalgebra of $\PTL_k(\delta)$ spanned by $\{\bar{d} \mid d \in
\D_k(k)\}$ is isomorphic to $\TL_k(\delta-1)$. For $\lambda\in
\Lambda$ with $|\lambda|=k$, the action of $\PTL_k(\delta)$ on
$\PTL^\lambda$, when restricted to that subalgebra, is isomorphic to
$\TL^\lambda$ as a $\TL_k(\delta-1)$-module. 
\end{rmk}

\section{$\UU_q(\gl_2)$ and $\UU_q(\fraksl_2)$}\label{s:quantum}%
\noindent
We assume henceforth that $\Bbbk$ is a field and that $0 \ne q \in
\Bbbk$ is not a root of unity. Let $\UU=\UU_q(\gl_2)$ be the quantized
enveloping algebra of the general linear Lie algebra $\gl_2$. By
definition, $\UU$ is the associative algebra with $1$ generated by
$E$, $F$, $K_i^{\pm1}$ (for $i=1,2$) subject to the defining relations
\begin{equation}
\begin{alignedat}{-1}
  & K_1K_2 = K_2K_1, && K_iK_i^{-1} = K_i^{-1}K_i = 1~(i=1,2)\\
  & K_1 E K_1^{-1} = q E, && K_2 E K_2^{-1} = q^{-1} E \\
  & K_1 F K_1^{-1} = q^{-1} F, && K_2 F K_2^{-1} = q F \\
  & EF - FE = \frac{K-K^{-1}}{q-q^{-1}}, && \text{where $K:= K_1K_2^{-1}$}. 
\end{alignedat}
\end{equation}
Furthermore, $\UU$ is a Hopf algebra with coproduct $\Delta: \UU \to
\UU \otimes \UU$ and counit $\epsilon: \UU \to \Bbbk$ given on
generators by
\begin{equation}\label{e:coprod}
\begin{alignedat}{-1}
  \Delta(E) &= E \otimes K + 1 \otimes E\\
  \Delta(F) &= F \otimes 1 + K^{-1} \otimes F\\
  \Delta(K_i) &= K_i \otimes K_i~(i=1,2)\\
    \epsilon(E) = \epsilon(F) &= 0, \quad \epsilon(K_i) = 1~(i=1,2).
\end{alignedat}
\end{equation}
The subalgebra of $\UU=\UU_q(\gl_2)$ generated by $E,F, K^{\pm1}$ is
the quantized enveloping algebra $\UU_q(\fraksl_2)$.

\begin{rmk}
The coproduct $\Delta$ defined in \eqref{e:coprod} is the one used in
\cite{BH}. It differs from the usual one in
\cites{Lusztig,Jantzen}. One could use either convention in this
paper, but for the sake of consistency, we stick with the choice made
in \cite{BH}.
\end{rmk}

We refer to \cite{Jantzen}*{Chap.~2} for basic facts about the
representation theory of $\UU_q(\fraksl_2)$. For each integer $n \ge
0$, by \cite{Jantzen}*{Thm.~2.6} there exist simple
$\UU_q(\fraksl_2)$-modules $L(n,+)$, $L(n,-)$ of dimension $n+1$. (In
characteristic $2$, $L(n,+) \cong L(n,-)$.) Any simple
$\UU_q(\fraksl_2)$-module of dimension $n+1$ is isomorphic to either
$L(n,+)$ or $L(n,-)$, and $L(n,+)$ is of type $\mathbf{1}$.

\begin{lem}
Suppose that $\Bbbk$ is a field and that $0 \ne q \in \Bbbk$ is not a
root of unity.  For any $\lambda = (\lambda_1,\lambda_2)$ in $\Z
\times \Z$ with $\lambda_1 - \lambda_2 \ge 0$, there is a unique
$\UU_q(\gl_2)$-module $V(\lambda)$ such that $V(\lambda) \cong
L(\lambda_1 - \lambda_2, +)$ as $\UU_q(\fraksl_2)$-modules, and
\[
K_i v_+ = q^{\lambda_i} \, v_+ \text{ for } i = 1,2
\]
where $v_+$ is a highest weight vector of weight $\lambda_1 -
\lambda_2$ for the $\UU_q(\fraksl_2)$-module structure.
\end{lem}

\begin{proof}
This follows from the results in \cite{Jantzen}*{Chap.~2}, using the
fact that any $\UU_q(\gl_2)$-module is also (by restriction) a
$\UU_q(\fraksl_2)$-module. (See also \cite{KS}*{\S7.3}.) 
\end{proof}

Under the same hypotheses, every finite dimensional
$\UU_q(\gl_2)$-module $M$ is semisimple. Furthermore, if $M$ is of
type $\mathbf{1}$, it is isomorphic to a direct sum of modules of the
form $V(\lambda)$ as in the above lemma.

The ``polynomial'' $\UU_q(\gl_2)$-modules are direct sums of
$V(\lambda)$ such that $\lambda=(\lambda_1,\lambda_2)$ belongs to
$\Z_{\ge 0} \times \Z_{\ge 0}$ and $\lambda_1 \ge \lambda_2$. If $M$
is a simple polynomial $\UU_q(\gl_2)$-module, we may (and do) identify
its highest weight $\lambda$ with a partition of at most two parts. In
particular, if $\lambda = (n)$ is a partition of one part, then we
write $V(n)$ for $V(\lambda)$. Thus $V(n) \cong L(n,+)$ as
$\UU_q(\fraksl_2)$-modules.

Consider the simple $\UU_q(\gl_2)$-modules $V(0)$, $V(1)$. Then $V(0)
\cong \Bbbk$ is the trivial module, with $\UU_q(\gl_2)$ acting via the
counit $\epsilon$. This means that if $v_0$ is a chosen basis of
$V(0)$; then on $v_0$ the operators $E$, $F$ act as zero and each
$K_i$ acts as $1$. Fix a choice of $v_0$. Fix also a basis $\{v_1,
v_{-1}\}$ of weight vectors of $V(1)$, where $v_1$ has weight $(1,0)$
and $v_{-1}$ has weight $(0,1)$ as $\UU_q(\gl_2)$-modules, such that
the action of $E$, $F$ is given by
\begin{equation}
  Ev_1 = 0,\quad Ev_{-1} = v_1,\quad  Fv_1 = v_{-1},\quad Fv_{-1} = 0.
\end{equation}
In other words, with respect to the basis $\{v_1, v_{-1}\}$, the
matrices representing the action of $E$, $F$, $K_i$ are
\begin{equation*}
  E \to 
  \begin{bmatrix}
    0&1\\0&0
  \end{bmatrix},\quad
  F \to 
  \begin{bmatrix}
    0&0\\1&0
  \end{bmatrix},\quad
  K_1 \to 
  \begin{bmatrix}
    q&0\\0&1
  \end{bmatrix},\quad
  K_2 \to 
  \begin{bmatrix}
    1&0\\0&q
  \end{bmatrix}.
\end{equation*}
Following \cite{BH}, we set $V := V(0) \oplus V(1)$, with basis
$\{v_1, v_0, v_{-1}\}$.

We will need the following general fact about tensor powers of
bialgebra representations, for which we were unable to find a suitable
reference.  If $\UU$ is a bialgebra and $V$ a $\UU$-module, then
$V^{\otimes k}$ is a $\UU$-module for any $k>1$, with $u \in \UU$
acting on $V^{\otimes k}$ by
\[
\Delta^{(k)}: \UU \to \UU^{\otimes k},
\]
where $\Delta$ is the coproduct on $\UU$, lifted to $\Delta^{(k)}$
inductively by (for example) defining $\Delta^{(2)} = \Delta$ and
$\Delta^{(k+1)} = (\Delta \otimes 1^{\otimes(k-1)}) \Delta^{(k)}$ for
$k \ge 2$.

\begin{lem}\label{l:comult}
Let $V$ be a $\UU$-module, where $\UU$ is a bialgebra with
coproduct~$\Delta$. Suppose that $\psi$ is in $\End_\UU(V\otimes V)$;
that is, $\psi \Delta(u) = \Delta(u)\psi$, for any $u \in \UU$. Then
$1^{\otimes(i-1)} \otimes \psi \otimes 1^{k-1-i}$ commutes with the
action of $\Delta^{(k)}(u)$, for any $u \in \UU$.
\end{lem}

\begin{proof}
This is of course well known; we sketch a proof for completeness. It
follows by induction from the coassociativity axiom
\[
(\Delta \otimes 1)\Delta = (1 \otimes \Delta) \Delta
\]
that
\[
\Delta^{(k)} = (1^{\otimes a} \otimes \Delta \otimes 1^{\otimes b})
\Delta^{(k-1)}
\]
for any $a,b$ such that $a+b = k-2$. Taking $a = i-1$, $b = k-1-i$
proves the result.
\end{proof}

\section{Structure of $V \otimes V$ for $V = V(0)\oplus V(1)$}%
\label{s:VV}\noindent
We continue to assume in this section that $\Bbbk$ is a field and $0
\ne q \in \Bbbk$ is not a root of unity.  We wish to analyze the
structure of $V \otimes V$, as module for both $\UU_q(\gl_2)$ and
$\UU_q(\fraksl_2)$. We have
\begin{align*}
V \otimes V = (V(0) \otimes V(0)) \;&\oplus\; (V(0) \otimes
V(1))\\  &\oplus\; (V(1) \otimes V(0)) \;\oplus\;
(V(1) \otimes V(1))
\end{align*}
and since the first three direct summands on the right hand side are
respectively isomorphic to $V(0)$, $V(1)$, and $V(1)$, understanding
the structure of $V \otimes V$ reduces to understanding the structure
of $V(1) \otimes V(1)$.

To simplify notation, set $v_{i,j} := v_i \otimes v_j$. Then $\{
v_{i,j} \mid i,j \in \{1,0,-1\}\}$ is a basis of $V \otimes V$, and
$\{ v_{i,j} \mid i,j \in \{1,-1\}\}$ is a basis of $V(1) \otimes
V(1)$. A simple direct computation shows that $\{v_{1,1},
q^{-1}v_{1,-1} + v_{-1,1}, v_{-1,-1}\}$ is a basis of weight vectors
for a submodule of $V(1) \otimes V(1)$ isomorphic to $V(2)$.  In order
to pick a complement to this submodule, we observe that the vector
\begin{equation}\label{e:Z_0}
  Z_0 := -q v_{1,-1} + v_{-1,1}
\end{equation}
is orthogonal to the submodule, with respect to the standard bilinear
form on $V(1)$ extended to $V(1)\otimes V(1)$, and this property
uniquely determines $Z_0$ up to a scalar multiple. The line $\Bbbk
Z_0$ in $V(1)\otimes V(1)$ is isomorphic to the trivial module $V(0)$
as $\UU_q(\fraksl_2)$-modules. Since $K_i$ for $i=1,2$ both act as $q$
on $Z_0$, it follows that $\Bbbk Z_0 \cong V(1,1)$ as
$\UU_q(\gl_2)$-modules.  Hence
\begin{equation}\label{e:V1V1}
  V(1) \otimes V(1) \cong 
  \begin{cases}
    V(2) \oplus V(1,1) & \text{as $\UU_q(\gl_2)$-modules}\\
    V(2) \oplus V(0) & \text{as $\UU_q(\fraksl_2)$-modules}.
  \end{cases}
\end{equation}
It follows from \eqref{e:V1V1} that
\begin{equation}\label{e:VV-dec}
  V \otimes V \cong
  \begin{cases}
    V(0) \oplus 2V(1) \oplus V(2) \oplus V(1,1) &
    \text{as $\UU_q(\gl_2)$-modules} \\
    2V(0) \oplus 2V(1) \oplus V(2) & \text{as $\UU_q(\fraksl_2)$-modules}.
  \end{cases}
\end{equation}
From this it is immediate that
\begin{equation}
\begin{aligned}
  \dim \End_{\UU_q(\gl_2)}(V \otimes V) &= 1^2 + 2^2 + 1^2 + 1^2 = 7  \\
  \dim \End_{\UU_q(\fraksl_2)}(V \otimes V) &= 2^2 + 2^2 + 1^2  = 9. 
\end{aligned}
\end{equation}
This dichotomy illustrates the error in \cite{BH}, where it is
implicitly assumed that the centralizers are the same in both cases.

The first decomposition given in equation \eqref{e:V1V1} is a special
case of the decomposition:
\begin{equation}\label{e:Pieri}
  V(1) \otimes V(\lambda) \cong \textstyle \bigoplus_\mu V(\mu)
\end{equation}
where $\lambda$, $\mu$ are partitions of at most two parts and $\mu$
varies over the set of such partitions obtainable from $\lambda$ by
adding one box. The rule \eqref{e:Pieri} is itself a special case of
the standard Pieri rule.  By repeated application of \eqref{e:Pieri},
we now construct the Bratteli diagram (see Figure~\ref{Bratteli}) for
the centralizer algebra
$\mathcal{Z}_k(q):= \End_{\UU_q(\gl_2)}(V^{\otimes k})$.

\begin{figure}[h!]
\begin{center}
\begin{tikzpicture}[xscale=2.6*\UNIT, yscale=-5*\UNIT]
  \coordinate (0) at (0,0);
  \foreach \x in {0,...,1}{\coordinate (1\x) at (2*\x,2);}
  \foreach \x in {0,...,3}{\coordinate (2\x) at (2*\x,4);}
  \foreach \x in {0,...,6}{\coordinate (3\x) at (2*\x,6);}
  \foreach \x in {0,...,11}{\coordinate (4\x) at (2*\x,8);}
  \draw (0)--(10) (0)--(11);

  \draw (10)--(20) (10)--(21) (11)--(21) (11)--(22) (11)--(23);

  \draw (20)--(30) (20)--(31) (21)--(31) (21)--(32) (21)--(33);
  \draw (22)--(32) (22)--(34) (22)--(35);
  \draw (23)--(33) (23)--(35) ;

  \draw (30)--(40) (30)--(41) (31)--(41) (31)--(42) (31)--(43);
  \draw (32)--(42) (32)--(44) (32)--(45);
  \draw (33)--(43) (33)--(45);

  \draw (34)--(44) (34)--(46) (34)--(47);
  \draw (35)--(45) (35)--(47) (35)--(48) ;
\def\NULL{\footnotesize$\emptyset$}
\begin{scope}[every node/.style={fill=white}]
  \node at (0) {\NULL};

  \node at (10) {\NULL};
  \node at (11) {\PART{1}};

  \node at (20) {\NULL};
  \node at (21) {\PART{1}};
  \node at (22) {\PART{2}};
  \node at (23) {\PART{1,1}};

  \node at (30) {\NULL};
  \node at (31) {\PART{1}};
  \node at (32) {\PART{2}};
  \node at (33) {\PART{1,1}};
  \node at (34) {\PART{3}};
  \node at (35) {\PART{2,1}};

  \node at (40) {\NULL};
  \node at (41) {\PART{1}};
  \node at (42) {\PART{2}};
  \node at (43) {\PART{1,1}};
  \node at (44) {\PART{3}};
  \node at (45) {\PART{2,1}};
  \node at (46) {\PART{4}};
  \node at (47) {\PART{3,1}};
  \node at (48) {\PART{2,2}};
\end{scope}
\end{tikzpicture}
\end{center}
\caption{Bratteli diagram for $\End_{\UU_q(\gl_2)}(V^{\otimes k})$,
$k \le 4$}\label{Bratteli}
\end{figure}

Since $q$ is not a root of unity, the dimension of the simple
$\mathcal{Z}_k(q)$-module indexed by a partition $\lambda$ of at most
two parts is the number of paths from the top vertex to its label in
the Bratteli diagram. The sum of the squares of those dimensions is
the dimension of $\mathcal{Z}_k(q)$, as tabulated below:
\[\setlength{\arraycolsep}{3pt}
\begin{array}{c|ccccccccc|c}
  \;k\; &\;\emptyset\;
  &(1)&(2)&(1^2)&(3)&(2,1)&(4)&(3,1)&(2,2)\;
  & \;\dim \mathcal{Z}_k(q) \\ \hline
  0 & 1 &&&&&&&& &1 \\
  1 & 1&1&&&&&&& &2 \\
  2 & 1&2&1&1&&&&&& 7\\
  3 & 1&3&3&3&1&2&&& &33\\
  4 & 1&4&6&6&4&8&1&3&2&  183
\end{array}
\]
and they differ from the dimensions of
$\End_{\UU_q(\fraksl_2)}(V^{\otimes k})$ given in \cite{BH}*{Fig.~1}.
Notice that the dimension of $\mathcal{Z}_k(q)$ agrees with the
dimension in \eqref{e:PTL-dim} of the algebra $\PTL_k(\delta)$, at
least up to degree $4$.  We will prove in Theorem~\ref{t:SWD} that
they agree in general, when $\delta = 1\pm(q+q^{-1})$.

\section{Schur--Weyl duality for the Motzkin algebras}%
\label{s:SWD-Motzkin}\noindent
We continue to assume that $\Bbbk$ is a field, and $0 \ne q \in \Bbbk$
is not a root of unity. We endow $V = V(0) \oplus V(1)$ with the
standard nondegenerate bilinear form such that $\bil{v_i}{v_j} =
\delta_{i,j}$ for all $i,j = 1,0,-1$ and extend the form to $V\otimes
V$ in the natural way.  Let $\pi$ be the orthogonal projection of
$V(1) \otimes V(1)$ onto the line $\Bbbk Z_0 \cong V(0)$, where $Z_0 =
-q v_{1,-1} + v_{-1,1}$ is the invariant in equation
\eqref{e:Z_0}. With respect to the ordered basis $v_{1,1}, v_{1,-1},
v_{-1,1}, v_{-1,-1}$ the matrix of $\pi$ is
\[
A(\pi) = \frac{1}{q+q^{-1}} \,
\begin{bmatrix}
  0 \\ & q & -1 \\ & -1 & q^{-1} \\ &&&0
\end{bmatrix}
\]
in which omitted entries should be interpreted as zero, as
usual. Recall from \S\ref{ss:TL} that the action of the generator
$e_i$ in $\TL_k(\pm(q+q^{-1}))$ on $V(1)^{\otimes k}$ is given by the
operator $\pm(q+q^{-1})\, 1^{\otimes(i-1)} \otimes \pi \otimes
1^{\otimes(k-i-1)}$. This defines a faithful action of
$\TL_k(\pm(q+q^{-1}))$, and it is well known that
\begin{equation}\label{e:SWD-TL}
\TL_k(\pm(q+q^{-1})) \cong \End_{\UU_q(\fraksl_2)}(V(1)^{\otimes k})
\cong \End_{\UU_q(\gl_2)}(V(1)^{\otimes k}).
\end{equation}
In order to extend the above action to a faithful action of
$\M_k(1\pm(q+q^{-1}))$ on $V^{\otimes k}$, we first consider how to do
this for the case $k=2$.

The algebra $\M_2(1\pm(q+q^{-1}))$ is generated by the elements 
$r$, $l$, and $e$ defined in \S\ref{s:diagrams}. It has
basis consisting of the nine Motzkin diagrams:
\[
\shade{\oneonepic}\;, \quad \shade{\rpic}\;, \quad \shade{\lpic}\;,
\quad \shade{\oneppic}\;, \quad \shade{\ponepic}\;,
\quad \shade{\epic}\;, \quad \shade{\cappic}\;,
\quad \shade{\cuppic}\;,  \quad \shade{\sink}\;.
\]
In order from left to right, the above elements are expressible in
terms of the generators $r$, $l$, $e$ as follows:
\[
1,\; r,\; l,\; lr,\; rl,\;  e,\; re = le,\; er = el,\;
r^2 = l^2 = rer = lel. 
\]
Since (in the partition algebra) the elements $r$, $l$ satisfy the
identities
\begin{equation}
  r = p_1s = sp_2, \qquad l = sp_1 = p_2s
\end{equation}
where $s$ is the swap operator defined in \S\ref{s:diagrams} and
$p_1$, $p_2$ are projections onto $V(0) \otimes V$, $V \otimes V(0)$
respectively, it is natural to define the action of $r$, $l$ on basis
elements by
\begin{equation}\label{e:rl}
  r v_{i,j} = \delta_{j,0} v_{0,i}, \quad
  l v_{i,j} = \delta_{i,0} v_{j,0}
\end{equation}
for all $i,j \in \{1,0,-1\}$, as in \cite{BH}*{\S3.4}.

The lines $\Bbbk v_{0,0}$, $\Bbbk Z_0$ are isomorphic copies of the
trivial $\UU_q(\fraksl_2)$-module in $V \otimes V$, so it is natural
to let $e$ act as a projection onto some linear combination of the
form
\begin{equation}
  Y_0 := v_{0,0} + \alpha Z_0 = v_{0,0} + \alpha(-q v_{1,-1} +
  v_{-1,1}) \qquad (\alpha \ne 0).
\end{equation}
We demand that on restriction to $V(1) \otimes V(1)$ (resp., $V(0)
\otimes V(0)$) the action restricts to the Temperley--Lieb action
defined above (resp., the trivial action). This forces $\alpha \ne 0$
and implies that each $v_{i,j}$ is sent to a multiple $\beta_{i,j}
Y_0$ of $Y_0$. The multipliers $\beta_{i,j}$ are forced by our demands
to be
\[
\beta_{1,-1} = \mp\alpha^{-1},\quad \beta_{0,0} = 1,\quad \beta_{-1,1} =
\pm \alpha^{-1} q^{-1} 
\]
with all other $\beta_{i,j} = 0$. In other words, the matrix of the
action of $e$ on the $0$-weight space $(V\otimes V)_0$ with respect to
the ordered basis $v_{1,-1}$, $v_{0,0}, v_{-1,1}$ is given by
\begin{equation}\label{e:B-matrix}
B(\alpha, \pm) :=
\begin{bmatrix}
  \pm q & -\alpha q & \mp 1 \\
  \mp \alpha^{-1} & 1 & \pm \alpha^{-1} q^{-1} \\
  \mp 1 & \alpha & \pm q^{-1}
\end{bmatrix} .
\end{equation}
The matrix $B(\alpha, \pm)$ satisfies the relation
\begin{equation}
B(\alpha,\pm)^2 = (1 \pm (q+q^{-1})) B(\alpha,\pm)
\end{equation}
so it is a scalar multiple of a projection in the above sense.

We now define two nondegenerate bilinear forms $\bil{-}{-}_t$,
$\bil{-}{-}_b$ on $V$ by the rules:
\begin{equation}\label{e:forms}
\begin{aligned}
  \bil{v_1}{v_{-1}}_t &= -\alpha q, \quad &
  \bil{v_0}{v_{-0}}_t &= 1, \quad &
  \bil{v_{-1}}{v_1}_t &= \alpha \\
  \bil{v_1}{v_{-1}}_b &= \mp \alpha^{-1}, \quad &
  \bil{v_0}{v_{-0}}_b &= 1, \quad &
  \bil{v_{-1}}{v_1}_b &= \pm \alpha^{-1} q^{-1}
\end{aligned}
\end{equation}
with all other $\bil{v_i}{v_j}_t = 0$, $\bil{v_i}{v_j}_b = 0$. It is
worth noticing that the nonzero values of $\bil{v_i}{v_j}_t$ (resp.,
$\bil{v_i}{v_j}_b$) are encoded in the middle column (resp., middle
row) of the matrix $B(\alpha, \pm)$.

Just as in \cite{BH}*{\S3.4}, the forms may be applied to give
explicit formulas for the action of $\M_k(1\pm(q+q^{-1}))$ on
$V^{\otimes k}$, as follows.  Given a Motzkin $k$-diagram $d$ and
$i_\alpha$, $j_\alpha$ in $\{-1,0,1\}$ for $\alpha = 1, \dots, k$,
label the top row vertices of $d$ from left to right with basis
elements $v_{j_1},\ldots, v_{j_k}$ and similarly label the bottom row
vertices with $v_{i_1},\ldots, v_{i_k}$. The blocks of $d$ are either
isolated vertices or edges.  Then:
\begin{equation}\label{e:M-action}
d (v_{i_1} \otimes \cdots \otimes v_{i_k}) = \sum_{j_1, \ldots, j_k}
(d)_{i_1, \ldots, i_k}^{j_1, \ldots, j_k}\ v_{j_1} \otimes \cdots
\otimes v_{j_k}
\end{equation}
defines the action of $d$, where the scalar $(d)_{i_1, \ldots,
  i_k}^{j_1, \ldots, j_k}$ is the product of the weights taken over
the various blocks of $d$. The weight $(\beta)_{i_1, \ldots,
  i_k}^{j_1, \ldots, j_k}$ of a labeled block $\beta$ of $d$ is
\begin{center}\renewcommand{\arraystretch}{1.3}%
\begin{tabular}{c|l}
  $\delta_{i,0}$ & \rule{0pt}{10pt}
  if $\beta$ is an isolated vertex labeled by $v_i$.
  \\ \hline
  $\delta_{i,j}$ & \rule{0pt}{10pt}
  if $\beta$ is a vertical edge with endpoints labeled by
  $v_i$, $v_j$.
  \\ \hline
  $\bil{v_i}{v_j}_t$ & \rule{0pt}{18pt}
  \begin{minipage}{3.8in}
    if $\beta$ is a horizontal top edge with left and right
    endpoints labeled by $v_i$ and $v_j$, resp.
  \end{minipage}\\ \hline
  $\bil{v_i}{v_j}_b$ & \rule{0pt}{18pt}
  \begin{minipage}{3.8in}
    if $\beta$ is a horizontal bottom edge with left and right
    endpoints labeled by $v_i$ and $v_j$, resp.
  \end{minipage}
\end{tabular}
\end{center}
Here $\delta_{i,j}$ is the usual Kronecker delta function.

\begin{prop}\label{p:Motzkin-action}
  For any $\alpha \ne 0$, the above rules define a left action of
  $\M_k(\delta)$ on $V^{\otimes k}$, where $\delta = 1 \pm
  (q+q^{-1})$.
\end{prop}

\begin{proof}
This is proved by the same argument as the proof of
\cite{BH}*{Prop.~3.29}. In that argument, the notations $v_1^* =
v_{-1}$, $v_0^* = v_0$, $v_{_1}^* = v_1$ were defined. The argument
depends on the following properties of the forms:
\[
\begin{aligned}
  \bil{v_a^*}{v_a}_b \bil{v_a}{v_a^*}_t &= 1,\quad \bil{v_a}{v_a^*}_b
  \bil{v_a^*}{v_a}_t &= 1\\
  \bil{v_a}{v_a^*}_t  \bil{v_a}{v_a^*}_b &=
  \begin{cases}
    -q & \text{ if } a = 1\\
     1 & \text{ if } a = 0\\
    -q^{-1} & \text{ if } a = -1
  \end{cases}
\end{aligned}
\]
and only on those properties. One checks that for any $\alpha \ne 0$,
all of these properties still hold when we use equation
\eqref{e:forms} to define the forms. 
\end{proof}

The following is a slight generalization of results proved in
\cite{BH}*{\S3.4}.

\begin{thm}\label{t:Motzkin-SWD}
Suppose that $\Bbbk$ is a field, $0 \ne q \in \Bbbk$ is not a root of
unity, and $0 \ne \alpha \in \Bbbk$. For any $i = 1, \dots, k-1$, let
$e_i$, $r_i$, $l_i$ act on $V^{\otimes k}$ in tensor positions $i$,
$i+1$ as the operator $1^{\otimes(i-1)} \otimes g \otimes
1^{\otimes(k-i-1)}$, where $g = e$, $r$, $l$ (as operators)
respectively. This extends to an action of $\M_k(1 \pm (q+q^{-1}))$
that commutes with the action of $\UU_q(\fraksl_2)$.  The
corresponding representation $\rho: \M_k(1 \pm (q+q^{-1}))
\to \End_\Bbbk(V^{\otimes k})$ is faithful, thus induces an algebra
isomorphism $\M_k(1 \pm (q+q^{-1}))
\cong \End_{\UU_q(\fraksl_2)}(V^{\otimes k})$.
\end{thm}

\begin{proof}
The defining relations for $\M_k(\delta)$ are given in \cite{HLP}.  It
is a tedious yet elementary calculation to verify that our operators
$e_i$, $r_i$, $l_i$ satisfy precisely the same relations (for any
$\alpha \ne 0$). It suffices to do the calculation in
$M_3(\delta)$. We used a computer algebra system to create explicit
matrices for the generating operators and verified the defining
relations accordingly. It follows that when $\delta = 1 \pm
(q+q^{-1})$, the action determines a representation
\[
\rho: \M_k(1 \pm (q+q^{-1})) \to \End_\Bbbk(V^{\otimes k}).
\]
Another way to see this is to repeat the proof of
\cite{BH}*{Prop.~3.29} with the appropriate substitutions. One can
check by an elementary direct computation that $e$, $r$, and $l$ (as
operators on $V \otimes V$) commute with the action of
$\UU_q(\fraksl_2)$; here again a computer algebra system is useful. It
then follows from Lemma~\ref{l:comult} that the action of any of the
generators $e_i$, $r_i$, $l_i$ commutes with the action of
$\UU_q(\fraksl_2)$.  Finally, the proof of \cite{BH}*{Thm.~3.31} also
applies to our situation to show that $\rho$ is faithful, and the
result follows.
\end{proof}

\begin{rmk}
(i) Working with $\M_k(1-q-q^{-1})$, Benkart and Halverson \cite{BH}
  choose to define the action of~$e$ (which they denote by
  $\mathfrak{t}$) on $V \otimes V)_0$ in terms of the matrix
\[
B(q^{-1/2},-) =
\begin{bmatrix}
  -q & -q^{1/2} & 1 \\
  q^{1/2} & 1 & -q^{-1/2} \\
  1 & q^{-1/2} & -q^{-1}
\end{bmatrix} 
\]
and define their bilinear forms $\bil{-}{-}_t$, $\bil{-}{-}_b$
accordingly.  In other words, they are setting $\alpha = q^{-1/2}$ and
making a particular choice of sign.  Thus they implicitly assume that
$q^{1/2}$ exists in $\Bbbk$. Our analysis shows that this assumption
is avoidable by simply taking $\alpha = 1$ (or any other convenient
nonzero value).

(ii) If we work instead in $\M_k(1+q+q^{-1})$, where $e^2 =
(1+q+q^{-1}) e$, it is also possible to define the action of $e$ as a
multiple of the \emph{orthogonal} projection onto $Y_0$ (with respect
to the standard bilinear form). The matrix giving the action of $e$ on
$(V\otimes V)_0$ with respect to the same ordered basis as above is
\[
B(\pm q^{-1/2},+) =
\begin{bmatrix}
  q & \mp q^{1/2} & -1 \\
  \mp q^{1/2} & 1 & \pm q^{-1/2} \\
  -1 & \pm q^{-1/2} & q^{-1}
\end{bmatrix} 
\]
With this choice, the corresponding bilinear forms $\bil{-}{-}_t$,
$\bil{-}{-}_b$ become identical.
\end{rmk}

\section{Schur--Weyl duality for $\PTL_k(\delta)$}%
\label{s:SWD}\noindent
We now turn to $\PTL_k(\delta)$, with $\delta = 1 \pm(q+q^{-1})$,
continuing to assume that $0 \ne q \in \Bbbk$ is not a root of unity
and $\Bbbk$ is a field. This algebra acts faithfully on $V^{\otimes
  k}$ by restriction of the action of $\M_k(\delta)$.  Recall from
\eqref{e:tilde-hat} that 
\[
\e_i = \tilde{e}_i = (1-p_i) e_i (1-p_i) .
\]
We have shown in Theorem~\ref{t:generators} that $\PTL_k(\delta)$ is
generated by the $\e_i$, $r_i$, $l_i$ ($i \in [k-1]$). By
Lemma~\ref{l:comult}, it suffices to understand the action in the case
$k=2$. Let $\e$ be the operator on $V \otimes V$ given by the action
of $\e_1$ in $\PTL_2(\delta)$. In terms of Motzkin diagrams we have
the identity
\[
\e = \shade{\epic} \,-\, \shade{\cappic} \,-\,
\shade{\cuppic} \,+\, \shade{\sink} \;.
\]
An explicit calculation with the formulas in the preceding section
reveals that $\e$ acts as
\[
v_{1,-1} \mapsto \pm(q v_{1,-1} - v_{-1,1}), \quad v_{-1,1} \mapsto
\pm(-v_{1,-1} + q^{-1} v_{-1,1})
\]
with all other $v_{i,j} \mapsto 0$. This is independent of the choice
of $\alpha \ne 0$. We note that $\e^2 = \pm(q+q^{-1}) \e$.  In other
words, the restriction of $\e$ to $V(1) \otimes V(1)$ coincides with
the standard action of the generator $e$ in $\TL_2(\pm(q+q^{-1}))$, as
described at the beginning of \S\ref{s:SWD-Motzkin}.

Thus $\e_i$ acts on $V^{\otimes k}$ as the operator $1^{\otimes(i-1)}
\otimes \e \otimes 1^{(k-i-1)}$, for any $i$ in $[k-1]$.

\begin{prop}
The action of $\PTL_k(1 \pm(q+q^{-1}))$ on $V^{\otimes k}$ commutes with
the action of $\UU_q(\gl_2)$.
\end{prop}

\begin{proof}
By Lemma~\ref{l:comult} it suffices to check that the action of $\e$
on $V \otimes V$ commutes with that of $\UU_q(\gl_2)$. We already know
that it commutes with the action of $E$, $F$, so we only need to check
commutation with the action of $K_1$, $K_2$. When restricted to $\Bbbk
v_{1,-1} \oplus \Bbbk v_{-1,1}$, each $K_i$ acts as $q$, so $K_i$ acts
as $q$ times the identity operator on that subspace. The result
follows.
\end{proof}

With $\delta = 1 \pm (q+q^{-1})$, let $\varphi$ be the restriction of
the representation $\rho: \M_k(\delta) \to \End_\Bbbk(V^{\otimes k})$
to $\PTL_k(\delta)$. Thus
\[
\varphi: \PTL_k(\delta) \to \End_\Bbbk(V^{\otimes k})
\]
is an injective algebra morphism.

\begin{prop}\label{p:triangular}
Suppose that $d = d(A,t,B)$ belongs to $\D(k)$, so that $d = r_A t_0
l_B$ as in Lemma~\textup{\ref{l:Motzkin-factors}}, where $t_0 = t
\otimes \omega_{k-n}$.  Then $\bar{d} = r_A \bar{t}_0 l_B$.
\end{prop}

\begin{proof}
Observe that by the definition of the bar elements,
\[
\bar{t}_0 = \textstyle \prod_{i=1}^n (1-p_i) \; t_0 \; \prod_{i=1}^n
(1-p_i).
\]
Then considerations similar to those in the proof of Theorem
\ref{t:bar-mult} show that
\[
r_A \, \textstyle \prod_{i=1}^n (1-p_i) \; t_0 \; \prod_{i=1}^n (1-p_i)
\, l_B = \prod_{i \in A} (1-p_i) r_A t_0 l_B \prod_{i \in B}
(1-p_i)
\]
and the result follows.
\end{proof}

To proceed, write $1 = \dot{p} + (1-\dot{p})$, where $1 = 1_V$ is the
identity map on $V$ and $\dot{p}$ is the projection map $V \to
V(0)$. This decomposition (a sum of orthogonal idempotents) induces
the defining decomposition
\[
V = \dot{p} V \oplus (1-\dot{p})V = V(0) \oplus V(1).
\]
As $V(0)$, $V(1)$ are $\UU_q(\gl_2)$-submodules of $V$, the operators
$\dot{p}$, $1-\dot{p}$ commute with the action of
$\UU_q(\gl_2)$. Since $p_i$ acts as $\dot{p}$ on the $i$th tensor
factor and as identity in the remaining factors, we can apply the same
reasoning to the $i$th tensor position in $V^{\otimes k}$ to obtain
the decomposition
\[
V^{\otimes k} = p_iV^{\otimes k} \oplus (1-p_i)V^{\otimes k} =
V^{\otimes i-1} \otimes \big(V(0) \oplus V(1)\big) \otimes
V^{\otimes k-i} .
\]
Expanding the operator $1^{\otimes k} = (\dot{p}+
(1-\dot{p}))^{\otimes k}$ binomially produces an identity
\[
1^{\otimes k} = \textstyle \sum_{A \subseteq [k]} \, \prod_{i \in A} (1-p_i)
\, \prod_{j \in [k] \setminus A} p_j. 
\]
Applying the above expansion to the space $V^{\otimes k}$ produces the
decomposition
\begin{equation}\label{e:tens-dec}
V^{\otimes k} = \textstyle \bigoplus_{A \subseteq [k]} \, V[A]
\end{equation}
where, for a given subset $A$ of $[k]$,
$
V[A] := V_1 \otimes V_2 \otimes \cdots \otimes V_k
$
where $V_i = V(1)$ if $i \in A$ and $V_i = V(0)$ otherwise.

\begin{prop}\label{p:B-to-A}
Let $d \in \D(k)$ and let $(A,d',B)$ be the corresponding triple under
the bijection in Lemma \textup{\ref{l:triples}}. If $n = |A|=|B|$ then
the representation $\bar{d} \mapsto \varphi(\bar{d})$ maps $V[B]$ into
$V[A]$ and maps all other $V[B']$ with $B' \ne B$ to zero.
\end{prop}

\begin{proof}
Let $d = rtl$ be the factorization in Proposition~\ref{p:triangular},
so that $\bar{d} = r\bar{t}l$. Then $\varphi(\bar{d}) = \varphi(r)
\varphi(\bar{t}) \varphi(l)$. Furthermore, $\varphi(r)$ induces an
isomorphism $V[\{1,\dots,n\}] \to V[A]$ and $\varphi(r) = 0$ on all
$V[Y]$ such that $Y \ne \{1,\dots, n\}$. (The inverse map is obtained
by flipping the diagram $r$ upside down.)  Similarly, $\varphi(l)$
induces an isomorphism $V[B] \to V[\{1,\dots,n\}]$ and $\varphi(l) =
0$ on all $V[Y]$ such that $Y \ne B$. As $\varphi(\bar{t})$ induces a
map from $V[\{1,\dots, n\}] \cong V(1)^{\otimes n}$ into itself, and
is zero on all other components, the result follows.
\end{proof}

We are now ready to prove the following.

\begin{thm}\label{t:SWD}
Suppose that $\Bbbk$ is a field and $0 \ne q \in \Bbbk$ is
not a root of unity. Set $\delta = 1 \pm(q+q^{-1})$. Then
\[
\End_{\UU_q(\gl_2)}(V^{\otimes k}) \cong \PTL_k(\delta).
\]
Hence, $V^{\otimes k}$ satisfies Schur--Weyl duality with respect to
the commuting actions of $\UU_q(\gl_2)$, $\PTL_k(\delta)$.
\end{thm}

\begin{proof}
Since the actions commute, $\varphi(\PTL_k(\delta))$ is contained in
the commuting algebra $\End_{\UU_q(\gl_2)}(V^{\otimes k})$. The action
of $\PTL_k(\delta)$ is faithful, so the desired isomorphism will
follow once we show the inclusion is an equality. We do this by
comparing dimensions. By the $\UU_q(\gl_2)$-module decomposition in
equation \eqref{e:tens-dec}, we have
\[
\dim_\Bbbk \End_{\UU_q(\gl_2)}(V^{\otimes k}) = \sum_{A,B} \dim_\Bbbk
\Hom_{\UU_q(\gl_2)}(V[A], V[B])
\]
where the sum is over all pairs $(A,B)$ of subsets of $[k]$. By
classical Schur--Weyl duality, the simple $\UU_q(\gl_2)$-modules
appearing as constituents of $V[A] \cong V(1)^{\otimes n}$, where
$|A|=n$, are all indexed by partitions of $n$ with not more than two
parts. Thus $\Hom_{\UU_q(\gl_2)}(V[B], V[A]) = (0)$ unless $|A|=|B|$.
Furthermore, if $|A|=|B|=n$ for subsets $A,B$ of $[k]$, we have
\[
\Hom_{\UU_q(\gl_2)}(V[B], V[A])
\cong \End_{\UU_q(\gl_2)}(V(1)^{\otimes n}) \cong \TL_n(\delta-1)
\]
by Schur--Weyl duality for Temperley--Lieb algebras \eqref{e:SWD-TL},
so
\[
\dim_\Bbbk
\Hom_{\UU_q(\gl_2)}(V[B], V[A]) = \dim_\Bbbk \TL_n(\delta-1) =
\mathcal{C}_n
\]
where $\mathcal{C}_n$ is the $n$th Catalan number. Putting these
facts together yields the equality
\[
\dim_\Bbbk \End_{\UU_q(\gl_2)}(V^{\otimes k}) = \sum_{n=0}^k
\binom{k}{n}^2 \mathcal{C}_n
\]
which by equation \eqref{e:PTL-dim} agrees with the dimension of
$\PTL_k(\delta)$. This proves the first statement in the theorem. The
remaining claims then follow by standard facts in the theory of
semisimple algebras.
\end{proof}

\begin{cor}\label{c:SWD}
  Under the same hypotheses, we have the decomposition
  \[
  V^{\otimes k} \cong \textstyle \bigoplus_{\lambda \in \Lambda} V(\lambda)
  \otimes \PTL^\lambda
  \]
  as $(\UU_q(\gl_2), \PTL_k(\delta)$-bimodules, where the
  indexing set $\Lambda$ is the set of partitions of $n$ of not more
  than two parts, for $0 \le n \le k$, as in
  Section~\textup{\ref{s:reps}}.
\end{cor}

\begin{proof}
This is a standard fact in semisimple representation theory.
\end{proof}

\begin{rmk}
(i) It makes sense to set $q=1$ in $\PTL_k(1\pm(q+q^{-1}))$, thus
obtaining $\PTL_k(1 \pm 2)$. If the field $\Bbbk$ has characteristic
zero, the analogue of Theorem~\ref{t:SWD} holds. In particular,
\[
\End_{U(\gl_2)}(V^{\otimes k}) \cong \PTL_k(1\pm 2)
\]
where $U(\gl_2)$ is the ordinary universal enveloping algebra of
$\gl_2$.  There is of course also a version of Corollary~\ref{c:SWD}
for this situation.

(ii) If $q$ is a root of unity then Theorem~\ref{t:SWD} (but not
Corollary~\ref{c:SWD}) still holds, provided that $\UU_q(\gl_2)$ is
replaced by an appropriate $\Bbbk$-form, but the proof is very
different. One needs to work with the Lusztig ``integral'' form of the
quantized enveloping algebra and to appeal to the paper \cite{DPS},
which established a version of Jimbo's Schur--Weyl duality at roots of
unity.

(iii) The image of $\UU_q(\gl_2)$ in $\End_\Bbbk(V^{\otimes k})$ is
isomorphic to a generalized $q$-Schur algebra in type A, in the sense
of \cite{D}, defined by the set $\Lambda$.
\end{rmk}


We now derive explicit formulas for the action of $\bar{d}$,
$\tilde{d}$, where $d \in \D(k)$.  Recall from Section~\ref{s:alt}
that
\[
\bar{d} = \prod_{i \in \topp{d}} (1-p_i) \; d\; \prod_{i' \in
  \bott{d}} (1-p_i),\quad \tilde{d} = \prod_{i \in \topp{d}_H} (1-p_i)
\; d \; \prod_{i' \in \bott{d}_H} (1-p_i).
\]
In the representation on tensor space, the element $1-p_i$ is the
operator $1^{\otimes (i-1)}\otimes (1-\dot{p}) \otimes 1^{\otimes
  (k-i)}$, where $\dot{p}$ is projection onto $V(0)$ and hence
$1-\dot{p}$ projects onto $V(1)$. This observation gives the following
result.

\begin{prop}
Suppose that $d$ is in $\D(k)$. Given $i_\alpha$, $j_\alpha$ in the
set $\{-1,0,1\}$ for $\alpha = 1, \dots, k$, label the top row
vertices of $d$ from left to right with $v_{j_1},\ldots, v_{j_k}$ and
similarly label the bottom row vertices with $v_{i_1},\ldots,
v_{i_k}$. Then the action $\bar{d}$ on $V^{\otimes k}$ is
given by
\[
\bar{d} (v_{i_1} \otimes \cdots \otimes v_{i_k}) = \sum_{j_1,
  \ldots, j_k} (\bar{d})_{i_1, \ldots, i_k}^{j_1, \ldots,
  j_k}\ v_{j_1} \otimes \cdots \otimes v_{j_k}
\]
and similarly the action $\tilde{d}$ on $V^{\otimes k}$ is
given by
\[
\tilde{d} (v_{i_1} \otimes \cdots \otimes v_{i_k}) =
\sum_{j_1, \ldots, j_k} (\tilde{d})_{i_1, \ldots, i_k}^{j_1, \ldots,
  j_k}\ v_{j_1} \otimes \cdots \otimes v_{j_k}
\]
where the scalars $(\bar{d})_{i_1, \ldots, i_k}^{j_1, \ldots, j_k}$
and $(\bar{d})_{i_1, \ldots, i_k}^{j_1, \ldots, j_k}$ are the product
over the modified weights of the labeled blocks of $d$. The modified
weight of a block $\beta$ in $d$ is the same is its weight minus a
correction term of $\delta_{i,0}\,\delta_{j,0}$ applied to all (resp.,
all horizontal) edges of $d$ in computing the action of $\bar{d}$
$(resp., \tilde{d}$).
\end{prop}

\begin{proof}
This follows from the observation preceding the proposition and the
table of weights preceding Proposition \ref{p:Motzkin-action}, by
considering the cases separately.
\end{proof}

\appendix
\section{Semisimplicity criterion for $\TL_k(\pm(q+q^{-1}))$}\label{a:TL}
\noindent
The purpose of this Appendix is to highlight a precise semisimplicity
criterion for Temperley--Lieb algebras that arose in the work of
Vaughan Jones, following the exposition of \cite{GHJ}.  Assume that
$\Bbbk$ is a field. For $q \in \Bbbk$, let
\[
[n]_q := 1+q+q^2+\cdots+q^{n-1}
\]
in $\Z[q]$ be the usual quantum integer (regarded as an element of
$\Bbbk$) and define $[k]^!_q = [1]_q [2]_q \cdots [k]_q$. If $q \ne 0$
then the balanced quantum integer $\Brackets{n}_q$ in $\Z[q, q^{-1}]$
may be defined as
\[
\Brackets{n}_q := q^{-(n-1)} [n]_{q^2} .
\]
We also define $\Brackets{k}^!_q = \Brackets{1}_q
\Brackets{2}_q \cdots \Brackets{k}_q$. Here is the criterion.

\begin{thm}[\cite{GHJ}]\label{t:ss-TL}
  If $\Bbbk$ is a field and $0 \ne q \in \Bbbk$ satisfies
  $\Brackets{k}^!_q \ne 0$ then $\TL_k(\pm(q+q^{-1}))$
  is semisimple over $\Bbbk$.
\end{thm}

\begin{proof}
For $0 \ne \beta \in \Bbbk$, the Jones algebra $A_k(\beta)$ is the
algebra with $1$ on generators $u_1, \dots, u_{k-1}$ subject to the
defining relations
\begin{equation*}
  u_i^2 = u_i, \quad \beta u_i u_{i\pm 1} u_i = u_i, \quad u_iu_j =
  u_j u_i \text{ if } |i-j|>1.
\end{equation*}
By Prop.~2.8.5(a) in \cite{GHJ}, $A_k(\beta)$ is (split) semisimple over
$\Bbbk$ if
\[
P_1(\beta^{-1}) P_2(\beta^{-1}) \cdots P_{k-1}(\beta^{-1}) \ne 0,
\]
where the $P_n(x)$ are polynomials in $\Z[x]$ satisfying the recursion
$P_0(x)=1$, $P_1(x)=1$, and $P_{n+1}(x) = P_n(x)-xP_{n-1}(x)$ for all
$n \ge 1$. Choose $q$ in $\Bbbk$ such that $q \ne 0$, $q \ne -1$, and
$\beta = q+q^{-1}+2$. (Replace $\Bbbk$ by a suitable quadratic
extension if necessary.) By Prop.~2.8.3(iv) in \cite{GHJ},
\[
P_n(\beta^{-1}) = \frac{1+q+q^2 + \cdots + q^n}{(1+q)^n} =
\frac{[n+1]_q}{(1+q)^n}.
\]
Hence, $A_k(\beta) = A_k(q+q^{-1}+2)$ is semisimple over $\Bbbk$ if
$[k]^!_q \ne 0$.  Now, by setting $e_i = \delta u_i$ for all $i$ we
recover the defining relations \eqref{e:TL-rels} if and only if $\beta
= \delta^2$, so $A_k(\delta^2) \cong \TL_k(\delta)$. We conclude that
if $[k]^!_q \ne 0$ then $\TL_k(\pm(q^{1/2}+q^{-1/2}))$ is semisimple
over $\Bbbk$.  To obtain the final conclusion, we replace $q^{1/2}$ by
$q$. This has the effect of replacing $[k]^!_q$ by $\Brackets{k}^!_q$,
up to a power of $q$.
\end{proof}

\begin{rmk}
It makes sense to specialize $q$ to $1$ in Theorem~\ref{t:ss-TL}.
Then $\Brackets{k}^!_1 = k!$ is the ordinary factorial of $k$, and the
semisimplicity criterion coincides with the one appearing in Maschke's
theorem for finite symmetric groups. This is no accident, as
$\TL_k(\pm(q+q^{-1}))$ is a quotient of an appropriate Iwahori--Hecke
algebra of type A.
\end{rmk}


\end{document}